\documentclass[10pt]{amsart}

\usepackage{subcaption} 
\usepackage{mathscinet}
\usepackage[utf8]{inputenc} 
\usepackage{tikz} 
\usepackage{tikz-cd} 
\usepackage{amssymb, amsmath, mathtools, amsfonts, amsthm} 
\usepackage{thmtools, thm-restate} 

\usepackage{graphicx}
\usepackage{fullpage}
\usepackage{caption} 
\usepackage{latexsym} 
\usepackage[inline,shortlabels]{enumitem}
\setitemize{itemsep=4pt}
\setenumerate{itemsep=3pt}
\usepackage{multicol} 
\usepackage{color}
\usepackage{tcolorbox}
\usepackage{setspace}

\usepackage[hypertexnames=false]{hyperref}
\definecolor{headercolor}{RGB}{255,255,240}
\definecolor{mycitecolor}{RGB}{169,169,169}
\definecolor{mylinkcolor}{rgb}{1.0, 0.0, 0.5}
\definecolor{myurlcolor}{rgb}{1.0, 0.0, 0.5}
\usepackage{colortbl}

\hypersetup{colorlinks=true,urlcolor=myurlcolor,
citecolor=mycitecolor,linkcolor=mylinkcolor,
linktoc=page,breaklinks=true}

\usepackage{url} 
\usepackage{cleveref} 
\usepackage[numbers]{natbib} 
\usepackage{indentfirst} 
\usepackage[mathscr]{eucal} 
\usepackage{colonequals} 
\usepackage{soul} 
\usepackage{marvosym} 

\newcounter{counter}[section]

\numberwithin{equation}{section}
\newtheorem{theorem}[counter]{Theorem}
\newtheorem{lemma}[counter]{Lemma}
\newtheorem{corollary}[counter]{Corollary}

\newtheorem{proposition}[counter]{Proposition}

\theoremstyle{definition}
\newtheorem{definition}[counter]{Definition}
\newtheorem{example}[counter]{Example} 

\theoremstyle{remark} 

\newtheorem{remark}[counter]{Remark}

\newcommand{\cdef}[1]{\textsf{{\color{blue}#1}}} 
  
\newcommand{\RR}{\mathbf{R}} 
\newcommand{\ZZ}{\mathbf{Z}} 
\newcommand{\QQ}{\mathbf{Q}} 
\newcommand{\CC}{\mathbf{C}} 
\newcommand{\DD}{\mathbb{D}} 
\newcommand{\OO}{\mathscr{O}}
\newcommand{\SnC}[2]{\mathcal{S}(#1/#2)} 
\newcommand{\hatK}{\SnC{K}{\QQ}}
\newcommand{\hatKf}[1]{\mathcal{S}\paren{#1; \QQ}}
\newcommand{\hatKfm}[2]{\mathcal{S}^{#2}\paren{#1; \QQ}}
\newcommand{\T}{T} 
\renewcommand{\AA}{\mathbf{A}} 
\newcommand{\PP}{\mathbf{P}}
 
\newcommand{\Qbar}{\bar{\QQ}}
\newcommand{\FF}{\mathcal{F}} 
\newcommand{\FD}{\mathcal{F}} 
\newcommand{\HH}{\mathbf{H}} 
\renewcommand{\SS}{\mathbf{S}}

\newcommand{\Pone}{\PP^1}

\newcommand{\md}{\,\operatorname{mod}\,} 
\newcommand{\dd}{\,\mathrm{d}} 

\newcommand{\paren}[1]{\left( #1 \right)}
\newcommand{\brk}[1]{\left\lbrace #1 \right\rbrace} 
 
\newcommand{\abs}[1]{\left| #1 \right|}

\newcommand{\Ht}{\mathrm{Ht}}

\newcommand{\mfor}{\text{ for }} 

\newcommand{\PGL}{\mathrm{PGL}}
\newcommand{\PSL}{\mathrm{PSL}}

\newcommand{\GL}{\mathrm{GL}} 

\let\originalleft\left \let\originalright\right
\renewcommand{\left}{\mathopen{}\mathclose\bgroup\originalleft}
  \renewcommand{\right}{\aftergroup\egroup\originalright}

\DeclareMathOperator{\Nm}{Nm} 
\DeclareMathOperator{\Disc}{Disc}

\DeclareMathOperator{\Aut}{Aut} 
\DeclareMathOperator{\Gal}{Gal}

\DeclareMathOperator{\re}{re}
\DeclareMathOperator{\im}{im}

\DeclareMathOperator{\Spec}{Spec}

\DeclareMathOperator{\sgn}{sgn} 
\DeclareMathOperator{\Sym}{Sym}

\DeclareMathOperator{\vol}{vol}

\definecolor{santicolor}{rgb}{1.0, 0.0, 0.16}
\definecolor{rlocolor}{rgb}{0.0, 0.45, 0.73}
\definecolor{arulcolor}{rgb}{0.0, 0.45, 0.73}
\definecolor{ilacolor}{rgb}{0.24, 0.7, 0.44} 
\definecolor{aslcolor}{rgb}{0.67, 0.38, 0.8}
\definecolor{willcolor}{rgb}{0.25, 0.4, 0.96}
\definecolor{fabcolor}{rgb}{0.0, 0.8, 0.8}
\definecolor{kevincolor}{rgb}{0.5, 1.0, 0.83}

\date{\today}
\title{Counting number fields of fixed degree by their smallest defining polynomial}

\author{Santiago Arango-Pi{\~n}eros} \address{Department of Mathematics,
University of Massachusetts Amherst, Amherst, MA 01003, USA}
\email{santiago.arango.pineros@gmail.com}
\urladdr{\url{https://sarangop1728.github.io/}}

\author{Fabian Gundlach} \thanks{FG was supported by the Deutsche
Forschungsgemeinschaft (DFG, German Research Foundation) --- Project-ID 491392403
--- TRR 358 (Project A4).} \address{Universität Paderborn, Fakultät EIM, Institut
für Mathematik, Warburger Str.~100, 33098 Paderborn, Germany.}
\email{fabian.gundlach@uni-paderborn.de}
\urladdr{\url{https://fabiangundlach.org/}}

\author{Robert J. Lemke Oliver} \thanks{RJLO was partially supported by NSF
grant DMS-2200760 and by the Office of the Vice Chancellor for Research at the
University of Wisconsin-Madison with funding from the Wisconsin Alumni Research
Foundation.} \address{Department of Mathematics, University of
Wisconsin-Madison, 480 Lincoln Dr, Madison, WI 53706}
\email{lemkeoliver@wisc.edu}
\urladdr{\url{https://lemkeoliver.github.io/}}

\author{Kevin J. McGown} \address{Department of Mathematics, California State
University Chico, Chico, California, 95929-0525} \email{kmcgown@csuchico.edu}
\urladdr{\url{https://kmcgown.yourweb.csuchico.edu}}

\author{Will Sawin} \address{Department of Mathematics, Princeton University,
Fine Hall, Washington Road, Princeton, NJ 08540}
\email{wsawin@math.princeton.edu} \urladdr{\url{https://williamsawin.com/}}
\thanks{WS was partially supported by NSF grant DMS-2502029 and a Sloan
Research Fellowship.}

\author{Allechar Serrano L\'opez} \address{Department of Mathematical Sciences,
Montana State University, Bozeman, MT 59717}
\email{allechar.serranolopez@montana.edu} \urladdr{\url{https://allechar.org/}}
\thanks{ASL was supported by NSF grant DMS-2418528.}

\author{Arul Shankar} \address{Department of Mathematics, University of
Toronto, Toronto, ON, Canada} \email{ashankar@math.toronto.edu}
\urladdr{\url{https://www.math.utoronto.ca/ashankar/}}

\author{Ila Varma} \address{Department of Mathematics, University of Toronto,
Toronto, ON, Canada} \email{ila@math.toronto.edu}
\urladdr{\url{https://www.math.toronto.edu/~ila/}}

\begin{document}

\begin{abstract} When do two irreducible polynomials with integer coefficients
  define the same number field? One can define an action of
  $\GL_2 \times \GL_1$ on the space of polynomials of degree $n$ so that for any two
  polynomials $f$ and $g$ in the same orbit, the roots of $f$ may be expressed
  as rational linear transformations of the roots of $g$; thus, they generate
  the same field. In this article, we show that almost all polynomials of
  degree $n$ with size at most $X$ can only define the same number field as
  another polynomial of degree $n$ with size at most $X$ if they lie in the
  same orbit for this group action. (Here we measure the size of polynomials by
  the greatest absolute value of their coefficients.)

  This improves on work of Bhargava, Shankar, and Wang, who proved a similar
  statement for a positive proportion of polynomials. Using this result, we
  prove that the number of degree $n$ fields such that the smallest polynomial
  defining the field has size at most $X$ is asymptotic to a constant times
  $X^{n+1}$ as long as $n\geq 3$. For $n = 2$, we obtain a precise asymptotic of
  the form $\frac{27}{\pi^2} X^2$. \end{abstract}

\maketitle

\setcounter{tocdepth}{1}
\tableofcontents

\newpage

\section{Introduction}

The most natural way of constructing number fields is to take irreducible
polynomials $f\in\ZZ[x]$, and to adjoin one of their roots to $\QQ$. That is,
setting $K_f=\QQ[x]/(f(x))$. The primitive element theorem implies that every
number field arises in this way, and in fact arises infinitely often. The
degree of the number field $K_f$ obtained from $f$ is equal the degree of $f$.
The most natural height function $\Ht$ on the family of degree $n$ polynomials
with integer coefficients is simply the maximum of the coefficients. Therefore,
for $n\geq 2$, we obtain a natural height function on the family of degree $n$
number fields $K$ by setting $\Ht(K)$ to be the minimum of $\Ht(f)$ as $f$
varies over integer degree $n$ polynomials giving rise to $K$. We call this the
\cdef{$\PP^1$-height of $K$}. In this article, we prove an asymptotic for
$N_n(\Ht \leq X)$, the number of (isomorphism classes) of degree $n$ number fields
$K$ with $\Ht(K)\leq X$.
\begin{theorem}
  \label{thm:main} 
  The asymptotic size of the set of (isomorphism classes of) degree $n$ number
  fields $K/\QQ$ with $\Pone$-height less than $X$ is
  \begin{equation*}
      N_n(\Ht < X) \sim 
      \begin{cases}
          \frac{27}{\pi^2}\cdot X^2, & \text{ if } n = 2, \\
          C_n\cdot X^{n+1}, & \text{ if } n > 2,
      \end{cases}
  \end{equation*}
  for some constants $C_n>0$, as $X\to\infty$.
\end{theorem}

As might be expected, proving the above result requires some control over the
set of polynomials $f$ giving rise the the same number field $K$. To better
understand this set, it will be convenient to use the natural bijection between
polynomials $f_0x^n+f_1x^{n-1}+\cdots +f_n$ of degree $\leq n$ and binary
$n$-ic forms $f_0x^n+f_1x^{n-1}y+\cdots +f_ny^n$. The reason to do this is that,
while they carry the same information, the binary $n$-ic form perspective makes
more obvious the action of $\GL_2$ on the space of binary forms via linear
change of variables. We say that two integer-coefficient binary $n$-ic forms
$f$ and $g$ are \cdef{equivalent} if some nonzero rational multiple of $f$ lies
in the $\GL_2(\QQ)$-orbit of $g$. Then we have the following result.

\begin{theorem}\label{thm:abstract} The proportion of irreducible
  integer-coefficient binary $n$-ic forms $f$ of height $\leq X$ such that:
\begin{enumerate}[leftmargin=*, itemindent=*]
\item[(i)] there exists an irreducible integer-coefficient binary $n$-ic form
  $g$ of height $\leq X$ satisfying $K_f \cong K_g$, and
    \item[(ii)] $f$ and $g$ are {\em not} equivalent
\end{enumerate}
amongst all integer-coefficient binary $n$-ic forms with height $\leq X$ tends to
$0$ as $X$ tends to $\infty$. \end{theorem}

\begin{remark}
  Recall that a closed point on $\Pone_\QQ$ is the vanishing locus of an
  irreducible binary $n$-ic form with integer coefficients. From this point of
  view, Theorem~ \ref{thm:main} is a result about low degree points on the
  projective line, in the sense of \cite{Viray--Vogt_2025}. Indeed, we are
  counting degree $n$ closed points on $\Pone_\QQ$ of bounded height which give
  rise to the same residue field.
\end{remark}

\subsection{History and motivation}
A classical question in number theory is: {\em for an integer $n\geq 2$, how many
  degree $n$ number fields are there with discriminant bounded by $X$}? A
folklore conjecture predicts that this number $N_n(|\Disc| < X)$ is asymptotic
to some constant times $X$. The quadratic case is classical, and amounts to
counting squarefree integers. It is known for $n=3$ by the landmark result of
Davenport--Heilbronn, for $n=4$ by work of Cohen-Diaz y Diaz-Olivier
\cite{CDyDO} and Bhargava \cite{B-quartic_count}, and for $n=5$ by Bhargava
\cite{B-quintic_count}. For $n\geq 6$, this conjecture is open, and works have
instead focused on obtaining upper and lower bounds (see for example
\cite{Schmidt-numberfields, EV-number_fields_bound, LOT-numberfield_bound,
  Patil-numberfield_bound}).

In recent years, there has been a great deal of interest in ordering and
counting number fields by invariants other than the discriminant. Much of the
attention has focused on invariants obtained by modifying the $p$-power
contribution of a ramified prime $p$ to the discriminant, depending on the
conjugacy class of the inertia group. However, an important different type of
invariant was used by Schmidt \cite{Schmidt-numberfields} in his seminal work
determining upper bounds for $N_n(|\Disc|<X)$: namely, for a number field $K$,
he considered the height of the smallest algebraic integer $\alpha\in \OO_K$ which
generates $K$, where the height of $\alpha$ with minimal polynomial
$x^n+\sum_{i\geq 1} c_{n-i}x^i$ is $\max(|c_k|^{1/k})$. We call this the
\cdef{$\AA^1$-height of $K$} since it is comparable to the height function
defined by looking at embeddings $\Spec K \to \AA^1_\QQ$, and choosing the
minimal multiplicative height of the images.

The $\AA^1$-height has been used to derive several results on the family of
degree $n$ number fields ordered by discriminant. Improving on Schmidt's
result, Ellenberg and Venkatesh \cite{EV-number_fields_bound} use the
$\AA^1$-height to improve upon the best lower bound for $N_n(|\Disc| < X)$, and
this was improved still further, again using the $\AA^1$-height, in
\cite{BhargavaShankarWang1-2022}.

Furthermore, studying degree $n$ $S_n$-number fields via their $\AA^1$-heights
has led to a re-derivation of the Malle--Bhargava heuristic
\cite{Malle-DGp2,Bhargava-Mass_formula} for their asymptotic number when
ordered by discriminant~\cite{ST-heuristics}.

By replacing $\AA^1_\QQ$ with other varieties into which $\Spec K$ might embed,
we obtain different height functions on the family of degree $n$ fields. Many
works have indicated that this is a fruitful approach to the problem of
counting number fields ordered by their discriminants. Indeed, the best known
upper bounds on the number of degree $n$ number fields ordered by discriminants
are obtained in \cite{LOT-numberfield_bound} using an $\AA^r$-height approach
for $r>1$, while the best known lower bounds are proved by Patil
\cite{Patil-numberfield_bound} using a $\PP^1$-height approach.

Our two main results fit into this framework, and are the two most natural
questions concerning any height on the family of degree $n$ number fields. In
\cite[Remark 3.3]{EV-number_fields_bound}, Ellenberg and Venkatesh hypothesized
an order of growth for the number of degree $n$ number fields ordered by
$\AA^1$- height; this was proved in \cite{BhargavaShankarWang1-2022}. Theorem
\ref{thm:main} answers a stronger version of the $\PP^1$ analogue. Meanwhile a
weaker version of the $\AA^1$-analogue of Theorem \ref{thm:abstract} was proved
in \cite[Lemma 7.2]{BhargavaShankarWang1-2022}. However, the $\PP^1$-height
analogues of these two $\AA^1$-height results are significantly more difficult
to obtain for a very simple technical reason: when degree $n$ monic
integer-coefficient polynomials $f(x)$ are ordered by $\AA^1$-height, the rings
$R_f\colonequals\ZZ[x]/(f(x))$ for almost all (100\% of) $f$ are highly
``skewed'' (the lengths of the Minkowski basis are far apart from one another).
If $f$ is such a polynomial, and if $R_f$ and $R_g$ are isomorphic for some
polynomial $g$ in-equivalent to $f$, then the height of $g$ must necessarily be
much larger than the height of $f$.

The analogous statement about Minkowski bases is false when
integer-coefficients binary $n$-ic forms are ordered by height. In fact, most
binary $n$-ic forms correspond to rings whose Minkowski bases (apart from the
element $1$) have lengths very similar to each other. So we instead study the
$S_n$-closure of these rings. We generalize work of Bhargava--Satriano in the
monic case by explicitly constructing the $S_n$-closures of rings arising from
binary $n$-ic forms. We study the relative lengths of Minkowski bases of these
$S_n$-closures, and prove that for most $f$ these rings are highly skewed (even
though the rank $n$ rings are not!). We utilize this skewness in the
$S_n$-closures to prove our main results. In the rest of the introduction,
after providing some background, we give a summary of our proofs of the main
results.

\subsection{Background}
\label{sec:background} Let $n\geq 2$, and let
\cdef{$V_n \colonequals \Sym^n(2) \cong \AA^{n+1}$} denote the \cdef{ space of
  binary $n$-ic forms}. If $B$ is a (base) ring, \cdef{$V_n(B)$} will denote
the \cdef{set of binary $n$-ic forms with coefficients in $B$}. Elements
$f \in V_n(B)$ will be written as:
\begin{align*}
    f(x,y) = f_0x^n + f_1x^{n-1}y + \cdots + f_ny^n = [f_0, f_1, \dots, f_n].
\end{align*}
The left action of $\GL_2(B)$ on $B^2$ induces a right action of $\GL_2(B)$ on
$V_n(B)$ given by

\begin{equation}\label{eq:action-on-forms}
    f^\gamma(x,y) \colonequals f(ax + by, cx + dy), \quad \gamma = \begin{pmatrix}
        a & b \\ c & d
    \end{pmatrix}.
\end{equation}
Observe that if $\gamma \in \GL_2(B)$, and $P = [\beta:\alpha]$ is a root of
$f \in V_n(B)$, then $\gamma^{-1}\cdot P$ is a root of $f^\gamma$. We obtain a right action of
$B^\times\times\GL_2(B)$ on $V_n(B)$, where $(\lambda,\gamma)$ sends $f$ to
$\lambda f^\gamma$. This action factors through the quotient by the torus
$$
T(B) \colonequals \{(t^{-n},\left(\begin{smallmatrix}t\\&t\end{smallmatrix}\right))\mid t\in B^\times\}\,.
$$
We write $$G(B) \colonequals (B^\times\times\GL_2(B))/T(B).$$
In particular, we will be concerned with the rings $B = \ZZ, \QQ, \Qbar$. 

The \cdef{discriminant} $\Disc=\Disc_n\in\ZZ[V_n]$ is a degree $2n-2$ polynomial
which is a \cdef{relative invariant} for the action of $G(B)$ on $V_n(B)$.
Namely, for $f\in V_n(B)$ and $(\lambda,\gamma)\in G(B)$, we have
\begin{equation*}
\Disc(\lambda f^\gamma)=\lambda^{2n-2}\det(\gamma)^{n(n-1)}\Disc(f).
\end{equation*}

Given $0\neq f\in V_n(\Qbar)$, we can factor $f$ as a product of linear forms:
\begin{align*}
    f &= (\alpha_1x-\beta_1y)(\alpha_2x-\beta_2y)\cdot\dots\cdot(\alpha_nx-\beta_ny).
\end{align*}
If $f_0\neq0$, we can also write
\[
  f =
  f_0\left(x-\tfrac{\beta_1}{\alpha_1}y\right)\left(x-\tfrac{\beta_2}{\alpha_2}y\right)\cdot\dots\cdot\left(x-\tfrac{\beta_n}{\alpha_n}y\right).
\]
We define \cdef{the roots of $f$} as points on the projective line
$\PP^1(\Qbar)$, namely the points $P_i \colonequals [\beta_i:\alpha_i]$. For
$f$ as above, the discriminant of $f \in V_n(\Qbar)$ can be written in terms of
its factorization:
\begin{equation*}
  \Disc(f) \colonequals \prod_{i<j} (\alpha_i \beta_j - \beta_i \alpha_j)^2=f_0^{2n-2}\prod_{i<j}\Bigl(\frac{\beta_i}{\alpha_i}-\frac{\beta_j}{\alpha_j} \Bigr)^2.
\end{equation*}

Next we describe the connection between integer coefficient binary $n$-ic forms
and rank $n$ rings.

For $f \in V_n(\ZZ)$ with $f_0\neq 0$, let $\theta$ be the image of $x$ in
$K_f \colonequals \QQ[x]/f(x,1)$. We associate to $f$ the lattice
$R_f \subseteq K_f$ generated by $\zeta_0 \colonequals 1$ and
\begin{equation}
  \label{eq:zeta_k}
  \zeta_k \colonequals f_0\theta^k + \cdots + f_{k-1}\theta, \mfor{1 \leq k \leq n-1.}
\end{equation}
It turns out that $R_f$ is a ring of rank $n$, as shown in \cite[Proposition
1.1]{Nakagawa89}. Subsequent work of Wood \cite{Wood-binary} extends this
construction, associating a rank $n$ ring $R_f$ to every binary $n$-ic form
$f\in V_n(\ZZ)$ (even $f_0=0$). We will let $K_f$ denote $R_f\otimes\QQ$, a definition
which agrees with our previous definition of $K_f$ when $f_0\neq 0$. Moreover, we
have the discriminant equality $\Disc R_f = \Disc(f)$ (see \cite[Proposition
4]{Simon01}).
\begin{definition}
  A \cdef{binary ring of rank $n$} is a rank $n$ ring $R$ that is isomorphic to
  $R_f$ for some binary $n$-ic form $f\in V_n(\ZZ)$.
\end{definition}
This construction of $R_f$ from $f$ generalizes in a straightforward way from
$\ZZ$ to arbitrary base ring $B$: elements in $V_n(B)$ give rise to
$B$-algebras of rank $n$. Moreover, the $B$-algebra $R_f$ corresponding to
$f\in V_n(B)$ is isomorphic to $R_g$ for every $g\in V_n(B)$ in the
$G(B)$-orbit of $f$. Hence, we obtain a map
\begin{equation*}
    V_n(B)/G(B) \to \brk{\text{$B$-algebras of rank $n$}}.
\end{equation*}

\begin{remark}
  If $n = 3$ and $B=\ZZ\textnormal{ or }\QQ$ (or in fact, any base scheme) the
  above map is a bijection, and is due to the Levi \cite{Levi}, Delone--Faddeev
  \cite{DeloneFaddeev} correspondence (later generalized by Gan--Gross--Savin
  \cite{GanGrossSavin}). This ($n = 3$) is the only case in which all rank $n$
  rings are binary, i.e., arise from binary $n$-ic forms. In the case $n = 4$
  and $B = \mathbf Z$, Wood~\cite{Wood-2012} showed that the quartic rings that
  arise from binary quartic forms are those that have monogenic cubic resolvent
  ring.
\end{remark}

\begin{remark}
\label{rmk:infinite-fibers}
For $n\geq4$ and $B = \QQ$, the map is far from injective. In fact, any field
extension $K/\QQ$ of degree~$n$ arises from infinitely many $G(\QQ)$-orbits in
$V_n(\QQ)$. To show this, interpret the projective line over~$K$ as an
$n$-dimensional variety over $\QQ$. For any proper subfield $K'\subsetneq K$, the points
in $\PP^1(K')$ lie on a subvariety of smaller dimension. Any point $P$ in
$\PP^1(K)\setminus\bigcup_{K'\subsetneq K}\PP^1(K')$ generates the field $K$, so there is an
irreducible form $f\in V_n(\QQ)$ which has $P$ as a root, and which therefore
satisfies $K_f\cong K$. On the other hand, $G(\QQ)$ acts on the roots of forms in
$V_n(\QQ)$ through the $3$-dimensional group $\PGL_2(\QQ)$. Hence, the roots of
the forms in a $G(\QQ)$-orbit all lie in a $3$-dimensional subvariety. However,
the (Zariski dense) subset
$\PP^1(K) \setminus \bigcup_{K'\subsetneq K}\PP^1(K')$ of the $n$-dimensional projective line over
$K$ cannot be covered by finitely many $3$-dimensional subvarieties.
\end{remark}

We say that two elements of $V_n(\ZZ)$ are \cdef{$G(\QQ)$-equivalent} if they
lie in the same $G(\QQ)$-orbit. When $f$ is irreducible, then $K_f$ is a number
field. When $f$ is separable, $K_f \simeq K_{F_1}\times\cdots\times K_{F_r}$ is an étale
$\QQ$-algebra, where $f = F_1\cdots F_r$ is the factorization of $f$ over $\QQ$.

Bhargava and Satriano~\cite{BhargavaSatriano-2014}, building on earlier work of
Bhargava~\cite[\S2.1]{Bhargava-HCL3-2024}, define the $S_n$-closure of a rank
$n$ ring $R$ as follows. First consider the $n$-fold tensor product
$R^{\otimes n}$ and the ideal $I(R/\ZZ)$ generated by all elements of the
form \begin{equation*} s_j(a)-\sum_{1\leq i_1<\cdots<i_j\leq n}a^{(i_1)}\cdots a^{(i_j)}
\end{equation*}for all $a\in A$ and $j\in\{1,\ldots,n\}$, where $s_j(a)$ denotes $(-1)^j$ times the $(n-j)$-th coefficient of the characteristic polynomial of the $\ZZ$-module homomorphism $ R \to R$ given by $x\mapsto ax$, and $a^{(i)}$ denotes the tensor $1\otimes\cdots\otimes a\otimes\cdots\otimes 1$ in $R^{\otimes n}$, with $a$ in the $i$-th position and ones elsewhere.

When $f$ is monic, so that $R_f$ is monogenic, an explicit basis of the
$S_n$-closure of $R_f$ was given by Bhargava and Satriano~\cite[Theorem
16]{BhargavaSatriano-2014} using powers of the generator of $R_f$. Using this
basis, they computed the discriminant of the ring. For most polynomials $f$,
the discriminant is not much smaller than the product of the lengths of the
basis vectors, showing that the lengths of these basis vectors approximate the
successive minima of the lattice. Using this, one can see that the
$S_n$-closure of $R_f$ is highly skewed, just as $R_f$ is.

On the way to proving our main theorems, we generalize this to arbitrary $f$.
We construct a basis for the $S_n$-closure of $R_f$, compute the discriminant
of $R_f$, and for most $f$ show that the lengths of these basis vectors
approximate the successive minima, showing that the $S_n$-closure of $R_f$ is
highly skewed, though $R_f$ itself usually is not.

\subsection{Strategy of proof}
\label{sec:strategy-of-proof}

We now give a heuristic argument for \Cref{thm:main} and explain the strategy
of proof. We use that most of the $\asymp X^{n+1}$ forms $f\in V_n(\ZZ)$ of height at
most $X$ satisfy the following properties:
\begin{enumerate}[(a)]
\item \label{it:f-Sn} $f$ is irreducible over $\QQ$ with Galois group $S_n$ by
  Hilbert's irreducibility theorem, and thus gives rise to a number field $K_f$
  of degree $n$ with Galois group $S_n$. (See \Cref{lem:usually-full-galois}.)
\item \label{it:f-large-disc} $\Disc(f) \asymp X^{2(n-1)}$ since $\Disc(f)$ is a
  homogeneous polynomial of degree $2(n-1)$ in the coefficients of~$f$. (See
  \Cref{lem:usually-large-disc}.)
\item \label{it:sqfree-disc-approx} $\Disc(f)$ is close to squarefree, i.e.,
  not divisible by a large square, essentially by
  \cite{BhargavaShankarWang2-2022}. Hence, the corresponding ring $R_f$ has
  small index in the ring of integers of $K_f$ and the discriminant of $K_f$ is
  $\asymp X^{2(n-1)}$. (See \Cref{lem:usually-almost-squarefree}.)
\end{enumerate}

A folklore conjecture (which is known to hold for $n\in\{2,3,4,5\}$) predicts
that there are $\asymp X^{2(n-1)}$ number fields with discriminant $\asymp X^{2(n-1)}$.

For $n=2$, we therefore have many more forms ($\asymp X^3$) than possible number
fields ($\asymp X^2$), so it seems likely that most quadratic number fields whose
discriminant is in range (i.e., is the discriminant of some $f\in V_n(\RR)$ of
height at most $X$) will correspond to some forms $f\in V_n(\ZZ)$ of height at
most $X$. As the number of such number fields is
$\sim \frac{27}{\pi^2} \cdot X^2$, this would imply the case $n=2$ of \Cref{thm:main}.

To make this rigorous, we use equidistribution results of Duke \cite{Duke88} in
the imaginary quadratic case and Chelluri \cite{Chelluri2004} in the real
quadratic case (also discussed in \cite[p.~4]{ELMV12}), which imply the
equidistribution of the subsets
\[
	\{f / \sqrt{|d|} : f \in V_2(\ZZ)\textnormal{ with }\Disc(f) = d\}
	\qquad\textnormal{of}\qquad
	\left\{f \in V_2(\RR) : \Disc(f) = \sgn(d)\right\}
\]
for fundamental discriminants $d\to\pm\infty$. (See \Cref{sec:quadratic-case}.)

For $n=3$, Levi has shown that two forms giving rise to the same number field
are always $G(\QQ)$-equivalent. This strong statement is false for $n\geq4$. (See
\Cref{rmk:infinite-fibers}.) However, we conjecturally have many more possible
number fields ($\asymp X^{2(n-1)}$) than forms ($\asymp X^{n+1}$) of height
$X$, which suggests that two forms of height at most $X$ giving rise to the
same number field are ``usually'' $G(\QQ)$-equivalent. This heuristically
justifies the case $n\geq3$ of \Cref{thm:abstract}. A sieve theory argument can be
used to show that the number of $G(\QQ)$-equivalence classes of forms
$f\in V_n(\ZZ)$ of height at most $X$ is $\sim C_n\cdot X^{n+1}$ for some
$C_n>0$. (See \Cref{thm:counting-equivalence-classes}.) This heuristically
justifies the case $n\geq3$ of \Cref{thm:main}.

To make this rigorous, we apply the following method that was used in
\cite[section~5]{BhargavaShankarWang1-2022} to count number fields by their
$\AA^1$-height. Consider monic polynomials $f\in\ZZ[x]$ of degree $n$, almost all
of which are irreducible and therefore give rise to an order
$R_f = \ZZ[x]/(f(x))$ in a number field $K_f = \QQ[x]/(f(x))$. Denote the image
of $x$ in $R_f$ by $\theta_f$. One can show that two monic polynomials giving rise
to the same number field are usually equivalent up to
$f(x)\mapsto a^{-n} f(ax+b)$ with $a\in\QQ^\times$ and $b\in\QQ$ by looking at the successive
minima of the lattice $R_f$ (which usually has small index in the integer
lattice of $K_f$). Consider the filtration
\[
	\QQ = U^{0}(f) \subsetneq U^{1}(f) \subsetneq \cdots \subsetneq U^{n-1}(f) = K_f
\]
given by $U^{d}(f) = \langle 1,\theta_f,\dots,\theta_f^d\rangle_\QQ$. Most monic polynomials satisfy
properties analogous to \ref{it:f-Sn}, \ref{it:f-large-disc}, and
\ref{it:sqfree-disc-approx}. Together with Minkowski's second theorem, these
properties imply that every set $R_f\cap U^d(f)$ contains a $\QQ$-basis of
$U^d(f)$ (namely $1,\theta_f,\dots,\theta_f^d$) consisting of vectors that are much
shorter than all vectors in $R_f\setminus U^d(f)$. This implies that two polynomials
giving rise to the same number field usually have the same associated
filtrations. In particular, the subspaces
$U^{1}(f) = \langle 1,\theta_f\rangle_\QQ$ agree, which implies that the polynomials are related
by $f(x)\mapsto a^{-n} f(ax+b)$.

When counting number fields by $\PP^1$-height, one runs into the following
problem: For forms $f\in V_n(\ZZ)$ of height at most $X$, the successive minima
of $R_f\setminus\QQ$ are usually all of a similar order of magnitude. We can therefore
only make use of the filtration $\QQ\subsetneq K_f$, which is too weak to recover
specific information about the form $f$. We overcome this obstacle by
considering the $S_n$-closures $\SnC{R_f}{\ZZ} \subseteq \SnC{K_f}{\QQ}$ of the ring
$R_f$ and the field $K_f$, respectively. In
\Cref{sec:basis-of-closure,sec:successive-minima}, using our control of the
successive minima of $\SnC{R_f}{\ZZ}$, we construct a filtration
\[
	\QQ = \hatKfm{f}{0} \subsetneq \hatKfm{f}{1} \subsetneq \cdots \subsetneq \hatKfm{f}{n-1} = \SnC{K_f}{\QQ}
\]
such that for most $f$, every set $\SnC{R_f}{\ZZ}\cap\hatKfm{f}{d}$ contains a
$\QQ$-basis of $\hatKfm{f}{d}$ consisting of vectors that are much shorter than
all vectors in $\SnC{R_f}{\ZZ}\setminus\hatKfm{f}{d}$. It follows that two forms giving
rise to the same number field usually have the same associated filtrations,
which, finally, implies that the two forms are $G(\QQ)$-equivalent. (See
\Cref{sec:without-representation-theory}.)

\section{The quadratic case}
\label{sec:quadratic-case}

\subsection{Setup} Let $\DD \subset \ZZ$ be the set of \cdef{fundamental
  discriminants}, i.e., the set of discriminants of quadratic number fields. As
with any discriminant, if $d \in \DD$, then $d\equiv 0,1 \md 4$. If
$d \equiv 0 \md 4$, then $d = 4D$ for some squarefree integer $D$ satisfying
$D \equiv 2,3 \md 4$. If $d \equiv 1 \md 4$, then $d = D$ is squarefree. It is well known
that
\begin{equation*}
    \label{eq:quadratic-fields-by-discriminant}
    \begin{aligned}
      N_2^+(\Disc < X) \colonequals \# \brk{K/\QQ \textrm{ quadratic field}:  0 < \Disc(K) \leq X} &= \#\brk{d \in \DD : 0 < d \leq X} \sim \frac{3}{\pi^2}\cdot X, \\
      N_2^-(\Disc < X) \colonequals \# \brk{K/\QQ  \textrm{ quadratic  field}:  -X \leq \Disc(K) < 0} &= \#\brk{d \in \DD : -X \leq d < 0} \sim \frac{3}{\pi^2}\cdot X.
    \end{aligned}
\end{equation*}
We will leverage this result to show the following theorem.
\begin{theorem}
    \label{thm:quadratic-case}
    The number of quadratic number fields with $\Pone$-height bounded by $X$,
    as $X \to \infty$, is
    \begin{equation*}
      N_2(\Ht < X) \colonequals \# \brk{K/\QQ  \textnormal{ quadratic field}:  \Ht(K) \leq X} \sim \frac{27}{\pi^2}\cdot X^2.
    \end{equation*}
\end{theorem}

The main ingredients in the proof are Duke's equidistribution theorem
\cite[Theorem 1]{Duke88} (in the case of negative discriminants) and Chelluri's
generalization (in the case of positive discriminants). In fact, we don't need
the full power of equidistribution, since mere density suffices for our
purposes. We proceed to recall the precise statements of their results.

\subsection{Duke's theorem for negative discriminants}
\label{sec:duke-imaginary}
Let $d \in \DD$ be a negative fundamental discriminant. Consider the quadratic
imaginary number field $\QQ(\sqrt{d})$ embedded in the complex numbers, so that
$\sqrt{d}/i > 0$. Let $\HH$ denote the upper half plane, with its usual metric
$\dd s^2 = y^{-2}(\dd x^2 + \dd y^2)$. The group $\PSL_2(\RR)$ is the isometry
group of the hyperbolic plane $\HH$, and it acts by fractional linear
transformations. Denote by $\mathcal{Y}(1)$ the \cdef{modular orbifold}
$\PSL_2(\ZZ)\backslash \HH$, and let
\begin{equation}
    \label{eq:fundamental-domain}
    \FD \colonequals \brk{z \in \HH : \re(z) \in [-1/2,1/2), |z| > 1, \text{ or } |z| = 1, \re(z) \leq 0},
\end{equation}
denote the fundamental domain. The set of \cdef{Heegner points} associated to
the negative fundamental discriminant $d$ is the finite set
\begin{equation*}
  \Lambda_d \colonequals \brk{\tau_{[a,b,c]} \colonequals \dfrac{-b+\sqrt{d}}{2a} \in \FD : d = b^2-4ac, \; a,b,c \in \ZZ}.
\end{equation*}
  
Gauss proved that the size of $\Lambda_d$ is precisely $h(d)$, the class number of
$\QQ(\sqrt{d})$. Duke proves that the Heegner points equidistribute in a
natural way, when $\DD \ni d \to -\infty$.
\begin{theorem}[Duke]
    \label{thm:duke-imaginary}
    Let $\dd\mu = \tfrac3\pi\dd x\dd y/y^2$ be the
    invariant probability measure
    of $\FD$. Let $\Omega \subset \FD$ be convex
    \footnote{The geodesic between any two points of $\Omega$ is itself
      contained in $\Omega$.} with piecewise smooth boundary.
    Then, there exists $\delta > 0$ depending only
    on $\Omega$ such that
    \begin{equation*}
        \dfrac{\#(\Lambda_d \cap \Omega)}{\#\Lambda_d} = \mu(\Omega) + O(|d|^{-\delta}), \quad \text{as } d\in \DD \text{ and } d \to -\infty.
    \end{equation*}
\end{theorem}
\begin{corollary}
  \label{cor:duke-imaginary}
  For any nonempty open $\Omega \subset\FD$, there exists $N_\Omega > 0$ such that for every
  fundamental discriminant $d < -N_\Omega$, there exists a quadratic form
  $[a,b,c]\in V_2(\ZZ)$ with $\tau_{[a,b,c]} \in \Omega \cap \Lambda_d$.
\end{corollary}

\subsection{Heights of imaginary quadratic fields} Recall that a quadratic form
$f = ax^2 + bxy + cz^2 \in V_2(\ZZ)$ is called \cdef{definite} if it has negative
discriminant, \cdef{positive} when $a > 0$, and \cdef{primitive} when
$\gcd(a,b,c) = 1$. Moreover, a positive definite form $f$ is said to be
\cdef{reduced} if it satisfies
\begin{enumerate}[label=(\roman*)]
    \item $|b| \leq a \leq c$, and
    \item If $|b| = a$ or $a = c$, then $b \geq 0$.
\end{enumerate}
More geometrically, to any form $f = ax^2 + bxy + cz^2 \in V_2(\ZZ)$ with
$a \neq 0$ and discriminant $d$ we can associate the complex number
$\tau_{[a,b,c]} = (-b+\sqrt{d})/2a$. When $f$ is positive definite,
$\tau_{[a,b,c]} \in \HH$, and
$f = a(x-\tau_{[a,b,c]}y)(x-\bar\tau_{[a,b,c]}y)$. If in addition $f$ is reduced,
then $\tau_{[a,b,c]}$ is in the fundamental domain~(\ref{eq:fundamental-domain}).
Moreover, if $\Disc(f) = d <0$ is a fundamental discriminant, then $f$ is
primitive and $\tau_{[a,b,c]} \in \Lambda_d$.

The following lemma is essentially \cite[Theorem 1]{Ruppert98}.
\begin{lemma}
    \label{lem:imaginary-ht-estimate}
    For $d \in \DD$ negative, we have that
    \begin{enumerate}[label=(\roman*)]
    \item \label{it:lower-imag} $\sqrt{\tfrac{|d|}{4}} \leq \Ht\paren{\QQ(\sqrt{d})},$ and
    \item \label{it:upper-imag} $\Ht\paren{\QQ(\sqrt{d})} \leq \paren{1 +
        o(1)}\sqrt{\tfrac{|d|}{4}},$ as $d \to -\infty$.
    \end{enumerate}
\end{lemma}
\begin{proof}
  For any form $f = [a,b,c]\in V_2(\ZZ)$ such that $K_f \cong \QQ(\sqrt{d})$, the
  binary ring $R_f$ is an order in the ring of integers of $\QQ(\sqrt{d})$, and
  so $b^2 - 4ac = \Disc(f) = \Disc(R_f) = m^2\cdot d$ for some integer
  $m\geq1$. It follows that
    \begin{equation*}
        4\Ht(f)^2 \geq |4ac| = |b^2 - dm^2| = b^2 + |d|m^2 \geq |d|.
    \end{equation*}
    This proves~\ref{it:lower-imag}.

    To prove~\ref{it:upper-imag}, let $1/2 > \varepsilon > 0$. Let
    $\Omega = \Omega_\varepsilon$ be the interior of the square with lower right corner at
    $i$ given by:
    \begin{equation*}
        \brk{\tau \in \CC : -\varepsilon\leq \re(\tau) \leq 0,\  1 \leq \im(\tau) \leq 1 + \varepsilon} \subseteq \FD.
    \end{equation*}
    From \Cref{cor:duke-imaginary}, there exists $N_\varepsilon > 0$ such that if
    $d < 0$ is a fundamental discriminant with $|d| > N_\varepsilon$, then we can find a
    Heegner point $\tau_{[a,b,c]} \in \Lambda_d \cap \Omega$. Since
    \begin{equation*}
        \re(\tau_{[a,b,c]}) = -b/2a, \quad \im(\tau_{[a,b,c]}) = \sqrt{|d|}/2a,
    \end{equation*}
    it follows that
    \begin{equation*}
      a > 0, \quad 0 \leq \frac{b}{2a} \leq \varepsilon, \quad 1 \leq \frac{\sqrt{|d|}}{2a} \leq 1 + \varepsilon .
    \end{equation*}
    From the third inequality above, it follows that
    $|a| = a \leq \sqrt{|d|/4}$ and with the second inequality, it follows that
    $|b| = b \leq 2 \varepsilon a \leq a \leq \sqrt{|d|/4}$. Putting everything together,
    \begin{align*}
        |c| = c &= \dfrac{b^2 + |d|}{4a} \\
        &\leq \dfrac{(2\varepsilon a)^2 + |d|}{4a} \\
        &= \varepsilon^2 a + \dfrac{\sqrt{|d|}}{2a} \cdot \sqrt{|d|/4} \\
        &\leq \varepsilon^2 \sqrt{|d|/4} + \paren{1 + \varepsilon}\sqrt{|d|/4} = \paren{1 + \varepsilon + \varepsilon^2} \sqrt{|d|/4}.
    \end{align*}
    The result follows by observing that $\Ht\paren{\QQ(\sqrt{d})}$ equals the
    minimum value of $\max(|a|,|b|,|c|)$, as $f =[a,b,c]$ ranges over all forms
    $f \in V_2(\ZZ)$ with discriminant $d$.
\end{proof}

\subsection{The unit tangent bundle}
\label{sec:T1-HH}
Recall that the \cdef{unit tangent bundle} of a two-dimensional Riemannian
manifold $M$, denoted $T_1M$, is the closed submanifold of the tangent bundle
$TM$ consisting of pairs $(z, \vec v)$ where $\vec v\in T_zM$ is a unit tangent
vector based at $z\in M$. The projection $T_1M \to M$ is an $\SS^1$-bundle, where
$\SS^1\colonequals \brk{z \in \CC : \|z\| = 1}$ is the unit circle. In this
section, we recall how $\PSL_2(\RR)$ is identified with the unit tangent bundle
of the upper-half plane $\HH$.

\medskip Recall that $\PSL_2(\RR)$ acts transitively by isometries on $\HH$.
For every $\gamma \in \PSL_2(\RR)$ we have the commutative square
  \begin{equation*}
    \begin{tikzcd}
      T_1\HH \arrow[r, "\dot\gamma"] \arrow[d] & \T_1\HH \arrow[d] \\
      \HH \arrow[r, "\gamma"'] & \HH,
    \end{tikzcd}
  \end{equation*}
  where $\dot\gamma$ is the differential of $\gamma$. This diagram gives rise to an
  action of $\PSL_2(\RR)$ on $T_1\HH = \HH\times\SS^1$. Indeed, take
  $\gamma \in \PSL_2(\RR)$ and $(\tau, \vec v) \in \HH\times\SS^1$. Let
  $\alpha(t)$ be a unit-speed geodesic satisfying $\alpha(0) =\tau$ and
  $\dot\alpha(0) = \vec v$. A calculation shows that
\begin{equation*}
  \dot{(\gamma\alpha)}(t)= \dfrac{\dot\alpha(t)}{j(\gamma,\alpha(t))^2},
\end{equation*}
where $j(\gamma,z)=cz+d$ is the automorphy factor of
$\gamma=(\begin{smallmatrix}a & b\\ c & d\end{smallmatrix})$. It follows that the
action of $\PSL_2(\RR)$ on $\HH\times\SS^1$ is given by:
\begin{equation*}
    \gamma\cdot (\tau, \vec v) = \left(\gamma\cdot \tau, \dfrac{\|j(\gamma,\tau)\|^2}{j(\gamma,\tau)^2}\vec v\right).
\end{equation*}

This action is simply transitive, and we obtain an isomorphism
$\HH\times \SS^1 \cong \PSL_2(\RR)$ by fixing a base point to act on. We choose the
point $(i,i) \in \HH\times\SS^1$ and obtain:
  \begin{equation*}
    \PSL_2(\RR) \to \HH\times \SS^1, \quad \gamma \mapsto \gamma\cdot(i,i) = \paren{\gamma\cdot i, \dfrac{\|j(\gamma,i)\|^2}{j(\gamma,i)^2} i}.
  \end{equation*}

  Under this identification, the unit-speed geodesic $ie^t$ on $\HH$ connecting
  $0$ and $\infty$ corresponds to the geodesic
\begin{equation*}
  \RR \to \PSL_2(\RR), \quad t \mapsto
  \begin{pmatrix}
        e^{t/2} & 0 \\
        0 & e^{-t/2}
    \end{pmatrix}.
  \end{equation*}
  The \cdef{geodesic flow} on $\PSL_2(\RR)$ is the smooth group action of the
  additive group $g\colon \PSL_2(\RR)\times \RR \to \PSL_2(\RR)$ given by
  \begin{equation*}
    (\gamma, t) \mapsto g_t(\gamma) \colonequals \gamma\cdot
      \begin{pmatrix}
        e^{t/2} & 0 \\
        0 & e^{-t/2}
    \end{pmatrix}.
  \end{equation*}
  We also have the two \cdef{horocycle flows}
  $h^+,h^- \colon \PSL_2(\RR)\times \RR \to \PSL_2(\RR)$ given by
    \begin{align*}
    (\gamma,u) \mapsto h^+_u(\gamma) \colonequals \gamma\cdot
      \begin{pmatrix}
        1 & u \\
        0 & 1
      \end{pmatrix}, \qquad
    (\gamma,v)\mapsto h^-_s(\gamma) \colonequals \gamma\cdot
      \begin{pmatrix}
        1 & 0 \\
        s & 1
    \end{pmatrix}.
    \end{align*}
    \begin{figure}[ht]
      \centering
      \includegraphics[width=0.6\textwidth]{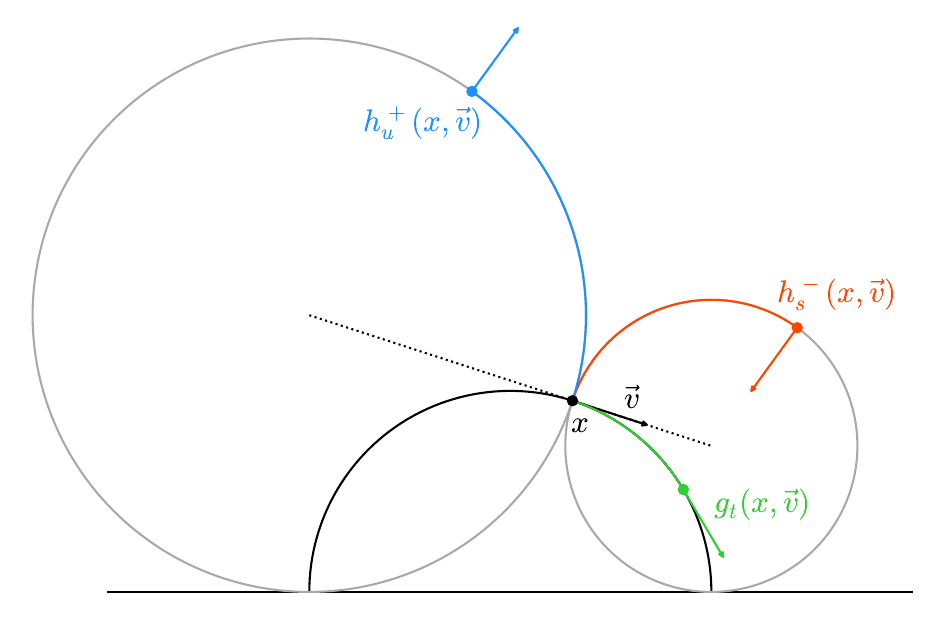}
      \caption{Geometric interpretation of the geodesic and horocycle flows on
        $T_1\HH =\PSL_2(\RR)$.}
      \label{fig:flows}
    \end{figure}

  \begin{remark}
    There are several ways to think about $T_1\mathcal{Y}(1)$, the unit tangent bundle of
    the modular surface (see \cite[Section 1.3]{ELMV12}). We will identify
    $T_1\mathcal{Y}(1)$ with $\PSL_2(\ZZ)\backslash \PSL_2(\RR)$, the space of lattices in
    $\RR^2$ with covolume one.
  \end{remark}

\subsection{Duke's theorem for positive discriminants} 
\label{sec:duke-real}
Let $d \in \DD$ be a positive fundamental discriminant. Consider the real
quadratic field $\QQ(\sqrt{d})$ embedded in $\CC$ so that $\sqrt{d}>0$. Suppose
that $[a,b,c] \in V_2(\ZZ)$ is a quadratic form with discriminant
$d = b^2 - 4ac > 0$ and $a\neq 0$. By replacing the form with its negative, if
necessary, we may assume that $a>0$. This form gives rise to a pair of real
numbers
\begin{equation*}
  \bar\tau_{[a,b,c]}\colonequals \dfrac{-b-\sqrt{d}}{2a} < \tau_{[a,b,c]} \colonequals \dfrac{-b + \sqrt{d}}{2a} \in \QQ(\sqrt{d})\subset \RR.
\end{equation*}

There is a unique unit-speed geodesic semi-circle in $\HH$ connecting
$\bar\tau$ to $\tau$ on $\HH$, given by the formula
\begin{equation*}
  t \mapsto g_t(\alpha_{[a,b,c]})\cdot i
\end{equation*}
where
\begin{equation*}
  \alpha_{[a,b,c]} \colonequals
  \sqrt{a/\sqrt{d}}
  \begin{pmatrix}
    \tau_{[a,b,c]} & \bar\tau_{[a,b,c]} \\
    1 & 1
  \end{pmatrix} \in \PSL_2(\RR).
\end{equation*}

By the discussion in \Cref{sec:T1-HH}, we can lift the geodesic
$\alpha_{[a,b,c]}(t)$ to the unit tangent bundle $T_1\HH = \PSL_2(\RR)$ via
\begin{equation*}
  \alpha_{[a,b,c]}(t) \colonequals g_t(\alpha_{[a,b,c]}).
\end{equation*}

Finally, $\alpha_{[a,b,c]}(t)$ descends to a closed geodesic
$[\alpha_{[a,b,c]}(t)] \colonequals \PSL_2(\ZZ)\alpha_{[a,b,c]}(t)$ on the unit tangent
bundle of the modular surface
$T_1\mathcal{Y}(1) = \PSL_2(\ZZ)\backslash \PSL_2(\RR)$. This geodesic orbit is compact and
depends only on the $\PSL_2(\ZZ)$-orbit of the form $[a,b,c]$. Taking the union
over all $\PSL_2(\ZZ)$-orbits of forms $[a,b,c]$ of discriminant $d$ geodesics,
we obtain the finite set
\begin{equation*}
  \Lambda_d \colonequals \bigcup_{[a,b,c]}
  \brk{[\alpha_{[a,b,c]}(t)]} \subset T_1\mathcal{Y}(1).
\end{equation*}
This collection of compact orbits of the geodesic flow carries a natural
probability measure, invariant under the geodesic flow, which we denote by
$\dd\mu_d$.

\begin{theorem}
    \label{thm:duke-real}
    Let $\dd\mu_{\mathrm{L}}$ be the probability Haar measure of
    $T_1\mathcal{Y}(1) = \PSL_2(\ZZ)\backslash\PSL_2(\RR)$. Let $\varphi\colon T_1\mathcal{Y}(1) \to \RR$ be a
    continuous and compactly supported function. Then, as $d \to \infty$ among
    the fundamental discriminants,
    \begin{equation*}
      \int_{\Lambda_d}\varphi(t)\dd\mu_d \to \int_{T_1\mathcal{Y}(1)} \varphi(t)\dd\mu_{\mathrm{L}}.
    \end{equation*}
  \end{theorem}

  The previous theorem is originally due to Chelluri \cite{Chelluri2004} and is
  stated here using the notation of \cite[Theorem~1.3]{ELMV12}. For our
  purposes, we need only the following corollary.

\begin{corollary}
  \label{cor:duke-real}
  For any nonempty open
  $\Omega \subset T_1\mathcal{Y}(1) = \PSL_2(\ZZ)\backslash\PSL_2(\RR)$, there exists
  $N_\Omega > 0$ such that for every fundamental discriminant $d > N_\Omega$, there
  exists a quadratic form $[a,b,c]\in V_2(\ZZ)$ with discriminant $d$ such that
  \begin{equation*}
    [g_t(\alpha_{[a,b,c]})] \in \Omega,
  \end{equation*}
  for some $t\in \RR$.
\end{corollary}

\subsection{Heights of real quadratic fields} We prove the version of
\Cref{lem:imaginary-ht-estimate} for positive fundamental discriminants. The
strategy of the proof is exactly the same: we find a convenient open
neighborhood of a ``binary form'' with minimal height (for a given
discriminant), say $f_{\min} = x^2 + xy - y^2$, and use Chelluri's result to
show that we can find binary forms $f$ of any discriminant $d\in\DD$ such that
$f/\sqrt d$ is close to $f_{\min}$, and hence $\Ht(f)$ is close to
$\sqrt{d}\cdot\Ht(f_{\min})$, as $d \to \infty$.
\begin{figure}[ht]
  \centering
  \includegraphics[width=0.5\textwidth]{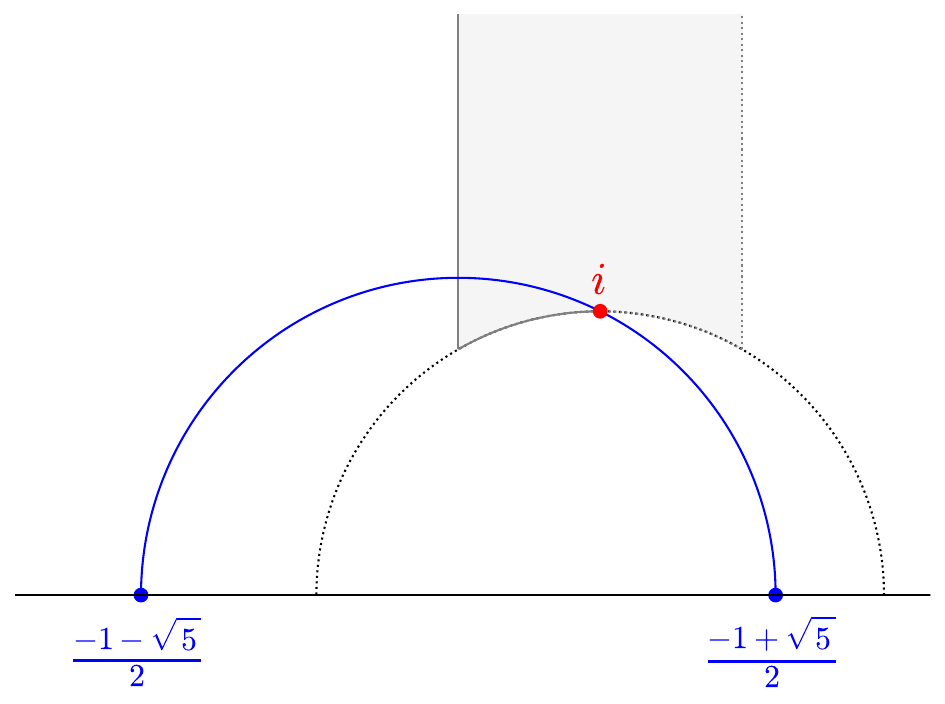}
  \caption{Geodesic corresponding to the quadratic form {\color{blue}$x^2+xy-y^2$} of
    discriminant $5$.}
\end{figure}

\begin{lemma}
    \label{lem:real-ht-estimate}
    For $d \in \DD$ positive, we have that
    \begin{enumerate}[label=(\roman*)]
    \item \label{it:lower-real}
      $\sqrt{\tfrac{d}{5}} \leq \Ht\paren{\QQ(\sqrt{d})},$ and
    \item \label{it:upper-real} $\Ht\paren{\QQ(\sqrt{d})} \leq \paren{1 + o(1)}\sqrt{\tfrac{d}{5}},$
      as $d \to \infty$.
    \end{enumerate}
\end{lemma}
\begin{proof}
  By definition, there exists $f = [a,b,c] \in V_2(\ZZ)$ with
  $K_f = \QQ(\sqrt{d})$ and $\Ht\paren{f}= \Ht\paren{\QQ(\sqrt{d})}$. Then
  $\Disc(f)=b^2-4ac$ must be $d$ times a perfect square. In particular, since
  $d$ is positive,
  \[\abs{d} =d \leq \Disc(f) = b^2-4ac \leq \abs{b}^2 + 4 \abs{a} \abs{c} \leq
    \Ht\paren{f}^2 + 4 \Ht\paren{f}^2= 5 \Ht(f)^2 =5
    \Ht\paren{\QQ(\sqrt{d})}^2\] which implies~\ref{it:lower-real}.

  To prove~\ref{it:upper-real}, let $1 > \varepsilon > 0$, and consider the matrix
  \begin{equation*}
    \alpha \colonequals \alpha_{[1,1,-1]} = \sqrt{1/\sqrt{5}}
    \begin{pmatrix}
      \tau & \bar\tau \\
      1 & 1
    \end{pmatrix} \in \PSL_2(\RR), \quad \text{ where } \quad \tau \colonequals \frac{-1+\sqrt{5}}{2}.
  \end{equation*}
  Take the following open neighborhood of $\alpha$ in $\PSL_2(\RR)$ defined using
  the flows $h^+,g$, and $h^-$
  \begin{equation*}
    U_\varepsilon \colonequals \brk{(h_s^-\circ g_t \circ h_u^+)(\alpha) :
    s,t,u \in (-\varepsilon, \varepsilon)}.
\end{equation*}
Let $\Omega_\varepsilon$ be the projection of $U_\varepsilon$ to
$\PSL_2(\ZZ)\backslash \PSL_2(\RR)$. By \Cref{cor:duke-real}, there exists
$N_\varepsilon > 0$ such that for every fundamental discriminant
$d > N_\varepsilon$, we can find a form $f = [a,b,c]\in V_2(\ZZ)$ of discriminant
$d$ with $[g_{t_0}(\alpha_{[a,b,c]})] \in \Omega_\varepsilon$ for some
$t_0\in \RR$. Thus, there exists some $\gamma \in \PSL_2(\ZZ)$ such that
$\gamma g_{t_0}(\alpha_{[a,b,c]}) = g_{t_0}(\gamma\alpha_{[a,b,c]}) \in U_\varepsilon$. But
$\gamma\alpha_{[a,b,c]} = \alpha_{[a',b',c']}$ where
$[a',b',c'] = f^{\gamma^{-1}}$, so without loss of generality we assume that
$g_{t_0}(\alpha_{[a,b,c]}) \in U_\varepsilon$. This means that there exist
$s,t,u \in (-\varepsilon,\varepsilon)$ such that
$g_{t_0}(\alpha_{[a,b,c]}) = (h_s^-\circ g_t \circ h_u^+)(\alpha)$. Expanding this expression, we
obtain
\begin{align*}
  \sqrt{a/\sqrt{d}}
  \begin{pmatrix}
    \tau_{[a,b,c]} & \bar\tau_{[a,b,c]} \\
    1 & 1
  \end{pmatrix}
  \begin{pmatrix}
    e^{t_0/2} & 0 \\
    0 & e^{-t_0/2}
  \end{pmatrix}
  &= \sqrt{1/\sqrt{5}}
  \begin{pmatrix}
      \tau & \bar\tau \\
      1 & 1
  \end{pmatrix}
  \begin{pmatrix}
      e^{t/2} + use^{-t/2} & ue^{-t/2}\\
      se^{-t/2} & e^{-t/2}
  \end{pmatrix} \\
  &= \dfrac{e^{-t/2}}{\sqrt[4]{5}}
  \begin{pmatrix}
    \tau(e^t + us) + \bar\tau s & \tau u + \bar\tau \\
    e^t + us + s & u + 1
  \end{pmatrix}
\end{align*}
Since the equality above is in $\PSL_2(\RR)$, the ratios of corresponding
entries of these matrices are equal, and the ratios between entries in a
vertical column are preserved by $g_{t_0}$, for small $\varepsilon>0$, we obtain the
identities
\begin{align*}
  \tau_{[a,b,c]} = \dfrac{-b+\sqrt{d}}{2a} &= \dfrac{\tau(e^t + us) + \bar\tau s}{e^t + us + s} = \tau+O(\varepsilon) = \dfrac{-1+\sqrt{5}}{2} + O(\varepsilon), \\
  \bar\tau_{[a,b,c]} =\dfrac{-b-\sqrt{d}}{2a} &= \dfrac{\tau u + \bar\tau}{u+1}                       = \bar\tau + O(\varepsilon) = \dfrac{-1-\sqrt{5}}{2}  + O(\varepsilon).
\end{align*}
Using these identities in combination with
\begin{equation*}
  \tau_{[a,b,c]}-\bar\tau_{[a,b,c]} = \dfrac{\sqrt{d}}{a}, \quad
  \tau_{[a,b,c]}+\bar\tau_{[a,b,c]} = -\dfrac{b}{a}, \quad
  \tau_{[a,b,c]}\bar\tau_{[a,b,c]} = \dfrac{c}{a},
\end{equation*}
one follows the same argument as in the proof of
\Cref{lem:imaginary-ht-estimate} to conclude that
$$
  \max\paren{\abs{a},\abs{b},\abs{c}} = (1 +o(1))\sqrt{d/5}\,.
$$
The result then follows by the inequality
$\Ht\paren{\QQ(\sqrt{d})} \leq \Ht([a,b,c])$.
\end{proof}
We note in passing that the previous lemma proves the limit conjectured
following the proof of Proposition~2 in Section~3 of~\cite{Ruppert98}.

\subsection{Proof of \texorpdfstring{\Cref{thm:quadratic-case}}{Theorem
    quadratic case}} We consider two counting functions for quadratic number
fields:
\begin{align*}
  N_2(\Ht < X) &\colonequals \#\brk{K \textnormal{ quadratic number field} : \Ht(K) < X} = \#\brk{d \in \DD : \Ht\paren{\QQ(\sqrt{d})} < X},\\
  N_2(\Disc < Y) &\colonequals \#\brk{K \textnormal{ quadratic number field} : |\Disc(K)| < Y} = \#\brk{d \in \DD : |d| < Y}.
\end{align*}
Partition $\DD = \DD^+ \sqcup \DD^-$ according to sign, and write
\begin{equation*}
    N_2(\Ht < X) = N_2^+(\Ht < X) + N_2^-(\Ht < X),
\end{equation*}
where
$N_2^{\pm}(\Ht < X) = \#\brk{d \in \DD^\pm : \Ht\paren{\QQ(\sqrt{d})} < X}$. From
\Cref{lem:real-ht-estimate}, we have that
\begin{align*}
  N_2^+(\Disc < 5X^2) = N_2^+(\Ht < X) + \#\brk{d \in \DD^+ : \sqrt{d/5} \leq X \leq \Ht\paren{\QQ(\sqrt{d})}} \sim N_2^+(\Ht <  X), \text{ as } X \to \infty.
\end{align*}

Similarly, \Cref{lem:imaginary-ht-estimate} gives
$N_2^-(\Disc,4X^2) \sim N_2^-(\Ht,X)$. We conclude that
\begin{equation*}
  N_2(\Ht,X) \sim N_2^+(\Disc,5X^2) + N_2^-(\Disc,4X^2) \sim \frac{3}{\pi^2}(5X^2) + \frac{3}{\pi^2}(4X^2) = \frac{27}{\pi^2}\cdot X^2 
\end{equation*}\hfill $\square$

\section{Special forms}
\label{sec:special-forms}

For any subset $S$ of $V_n(\ZZ)$, let $N(S,X)$ be the number of forms $f\in S$
with $\Ht(f)\leq X$. For $S\subseteq T\subseteq V_n(\ZZ)$ with
$T\neq\emptyset$, we say that $S$ has \cdef{density} $c$ inside $T$ if
$\lim_{X\to\infty} N(S,X)/N(T,X) = c$. We say that a binary $n$-ic form is
\cdef{primitive} if the greatest common divisor of the coefficients of $f$ is
$1$.

Let $V_n'(\ZZ)$ be the set of primitive binary $n$-ic forms $f \in V_n(\ZZ)$ such
that $K_f$ is a number field whose Galois closure has Galois group $S_n$.

We will show that for small $\varepsilon>0$ and large $s\geq1$, the following subsets have
small density in $V_n(\ZZ)$:
\[
B(\varepsilon) \coloneq \{f\in V_n(\ZZ) : |\Disc(f)|<\varepsilon\Ht(f)^{2(n-1)}\},
\]
\[
C(s) \coloneq \{f\in V_n(\ZZ) : a^2 \mid \Disc(f)\textnormal{ for some }a>s\}.
\]

For $B(\varepsilon)$, this will be done in \Cref{ss-small-discriminant}. For
$C(s)$, the crucial special case is when $a$ is a prime power, which will be
handled in \Cref{ss-small-prime} by $p$-adic integration methods. Using this,
the general case will be handled in \Cref{ss-large-square}.

\subsection{Forms with smaller Galois groups}

\begin{lemma}\label{lem:usually-full-galois}
  The set $V_n'(\ZZ)$ (defined above) has density $1$ inside the set of all
  primitive binary $n$-ic forms.
\end{lemma}
\begin{proof}
  For the binary $n$-ic forms $f$ with $f_0 \neq 0$, the number field $K_f$ is
  isomorphic to $\QQ[x]/f(x,1)$, so this is a direct application of Hilbert's
  irreducibility (see \cite[Theorem 2.1]{Cohen81}). Since $f_0 = 0$ is a
  hyperplane in $V_n$, this is enough.
\end{proof}

\subsection{Forms whose discriminant is small}\label{ss-small-discriminant}

\begin{lemma}
\label{lem:usually-large-disc}
We have $N(B(\varepsilon),X) = \beta(\varepsilon) X^{n+1} + O_\varepsilon(X^n)$ for
$X\geq1$ for some $\beta(\varepsilon)$ with $\lim_{\varepsilon\to0}\beta(\varepsilon)=0$.
\end{lemma}
\begin{proof}
  Let $B_\RR(\varepsilon)$ be the set of forms $f\in V_n(\RR)$ with
  $|\Disc(f)|<\varepsilon\Ht(f)^{2(n-1)}$. This set is invariant under scaling and
  semi-algebraic. By \cite{Davenport-1951}, the number of integer lattice
  points in this set therefore satisfies
  $N(B(\varepsilon),X)=\beta(\varepsilon)X^{n+1}+O_\varepsilon(X^n)$ with
  $\beta(\varepsilon) = \vol(B_\RR(\varepsilon)\cap\{\Ht\leq1\})$. Since
  $\bigcap_{\varepsilon>0} B_\RR(\varepsilon) = \{\Disc=0\}$ has measure $0$, we have
  $\beta(\varepsilon)\to0$ for $\varepsilon\to0$ by the dominated convergence theorem.
\end{proof}

\subsection{Forms whose discriminant is divisible by a large power of a small
  prime}\label{ss-small-prime}

For $K$ an \'{e}tale algebra over $\QQ_p$ (i.e., a product of finite field
extensions of $\QQ_p$), let $\OO_{K}$ be the set of elements in $K$ which are
integral over $\mathbf Z_p$ (i.e., the product of the rings of integers of the
field extensions).

For any $m\geq1$, we let
\[
  \PP^m(K) = \{(x_0,\dots,x_m)\in K^{m+1}\mid x_0,\dots,x_m\textnormal{ generate the
    unit ideal of }K\} / K^\times
\]
and denote its elements by $[x_0:\dots:x_m]$ as usual. Equivalently,
\[
  \PP^m(K) = \{(x_0,\dots,x_m)\in \OO_K^{m+1}\mid x_0,\dots,x_m\textnormal{ generate
    the unit ideal of }\OO_K\} / \OO_K^\times.
\]

The ring of integers $\OO_K$ carries a unique Haar
measure $\mu_{\OO_K}$ satisfying $\mu_{\OO_K}(\OO_K)=1$.
We denote the resulting
($\GL_m(\OO_K)$-invariant) standard measure on $\OO_K^m$ by $\mu_{\OO_K^m}$. We
define a measure $\mu_{\PP^m(K)}$ on $\PP^m(K)$ by
\[
	\mu_{\PP^m(K)}(A) = \mu_{\OO_K^{m+1}}\left(\left\{(x_0,\dots,x_m)\in\OO_K^{m+1}\;\middle|\;
        \begin{matrix}x_0,\dots,x_m\textnormal{ generate the unit ideal of
            $\OO_K$}\\\textnormal{ and } [x_0:\dots:x_m]\in
          A\end{matrix}\right\}\right)\Big/\mu_{\OO_K}(\OO_K^\times).
\]
This measure is $\GL_{m+1}(\OO_K)$-invariant and corresponds to the standard
measure on $\OO_K^m$ under the embedding
\[
	\OO_K^m \hookrightarrow \PP^m(K),\qquad
	(x_1,\dots,x_m) \mapsto [1:x_1:\dots:x_m].
\]
(In particular, the open subset
$\{ [1: x_1:\dots: x_m] \mid x_1,\dots, x_m \in \OO_K\} \cong \OO_K^m$ of
$\PP^m(K)$ has measure $1$.)

Let $\operatorname{Mor}(K,\overline{\QQ}_p)$ be the set of $\QQ_p$-algebra
homomorphisms from $K$ to a fixed algebraic closure $\overline{\QQ}_p$ of
$\QQ_p$. There is a norm map
$\operatorname{Nm} \colon \PP^1(K) \to \PP^n(\QQ_p)$ which sends the element
$[\beta : \alpha ] \in \PP^1(K) $ to $[a_0:\dots:a_n]$, where
\[a_i = (-1)^i \sum_{ \substack{ S \subseteq \operatorname{Mor}(K,\overline{\QQ}_p)\\
      \abs{S}=i}} \prod_{\sigma \in S} \sigma(\beta) \prod_{\sigma \in S^c} \sigma(\alpha) \] are the coefficients of
the binary $n$-ic form
$\prod_{\sigma\in\operatorname{Mor}(K,\overline{\QQ}_p)}(\sigma(\alpha)x-\sigma(\beta)y)$. (For
$\alpha=1$, these are the coefficients of the characteristic polynomial of $\beta$.)

The discriminant of a binary $n$-ic form is a homogeneous polynomial $\Disc$ in
its coefficients $a_0,\dots, a_n$. For $P=[a_0:\dots:a_n] \in \PP^n(K)$ where
$a_0,\dots,a_n\in \OO_{K}$ generate the unit ideal, we let $\abs{\Disc(P)}_p$ be
the $p$-adic absolute value of $\Disc(a_0x^n+\dots+a_ny^n)$, which does not
depend on the choice of representative.

Fix now a representative of each isomorphism class of \'{e}tale algebras of
degree $n$ over $\QQ_p$.

\begin{lemma}\label{change-of-variables} For any
  $\mu_{ \PP^n(\QQ_p)}$-integrable function $\psi$ on $\PP^n(\QQ_p)$, we have

  \[ \int_ {\PP^n(\QQ_p)} \psi d \mu_{\PP^n(\QQ_p)} = \sum_{\substack{K/\QQ_p\\ \deg n \\
        \text{\'etale}}} \frac{ \abs{\Disc(K/\QQ_p)}_p^{1/2}}{\abs{\Aut(K)}}
    \int_{\PP^1(K)} \abs{ \Disc(\operatorname{Nm}(P))}_p^{1/2} \psi(
    \operatorname{Nm}(P)) d\mu_{\PP^1(K)}(P).\]
\end{lemma} 

\begin{proof}
  Every $[f]\in\PP^n(\QQ_p)$ with non-vanishing discriminant lies in the image of
  the norm map $\Nm\colon\PP^1(K)\to\PP^n(\QQ_p)$ of exactly one étale
  $\QQ_p$-algebra $K$ (up to isomorphism), namely $K_f$. Since the set of
  $[f]\in\PP^n(\QQ_p)$ with vanishing discriminant has measure zero, it suffices
  to check that for every $\QQ_p$-algebra $K$, the integral over the image of
  the corresponding norm map is given by

\begin{equation}
\label{eq:change-of-variables}
\int_{\Nm(\PP^1(K))} \psi d\mu_{\PP^n(\QQ_p)}
=
\frac{ \abs{\Disc(K/\QQ_p)}_p^{1/2}}{\abs{\Aut(K)}}  \int_{\PP^1(K)} \abs{ \Disc(\operatorname{Nm}(P))}_p^{1/2} \psi( \operatorname{Nm}(P)) d\mu_{\PP^1(K)}(P).   
\end{equation}

The determinant of a map between two vector spaces over $\QQ_p$ of the same
dimension is well-defined if we choose a basis for these vector spaces. The
$p$-adic absolute value of the determinant depends only on the
$\mathbf Z_p$-lattice generated by the basis, and so is well-defined if we
choose a $\mathbf Z_p$-lattice in each vector space .The determinant of a map
between two vector spaces over $\QQ_p$ of the same dimension is well-defined if
we choose a basis for these vector spaces. The $p$-adic absolute value of the
determinant depends only on the $\mathbf Z_p$-lattice generated by the basis,
and so is well-defined if we choose a $\mathbf Z_p$-lattice in each vector
space. \Cref{eq:change-of-variables} follows from the general change of
variables formula for $p$-adic integration once we prove the following claim.

\textbf{Claim:} The norm map is $\abs{\Aut(K)}$-to-1 onto its image outside a
set of measure $0$ and the determinant of the Jacobian of the norm map, defined
using the $\mathbf Z_p$-lattice of the tangent space of projective space
$\PP^1(K)$ arising from the tangent space of $\PP^1_{\OO_K}$ and the
$\mathbf Z_p$-lattice of the tangent space of projective space $\PP^n(\QQ_p)$
arising from the tangent space of $\PP^n_{ \mathbf Z_p}$, is
$\abs{\Disc(K/\QQ_p)}_p^{1/2}\abs{ \Disc(\operatorname{Nm}(P))}_p^{1/2}$.

Let us first see that the map is $\lvert\Aut(K)\rvert$-to-$1$. There is a
natural action of $\Aut(K)$ on $\PP^1(K)$. To show the map is
$\lvert\Aut(K)\rvert$-to-$1$, it suffices to show the inverse image of any
point consists of a single simply transitive $\lvert\Aut(K)\rvert$-orbit. An
orbit only fails to be simply transitive if an element in the orbit is fixed by
some automorphism, which happens if and only if it is contained in $\PP^1(K')$
for some proper subfield $K'$ of $K$. This forms a lower-dimensional subvariety
and hence has measure $0$. Outside the vanishing locus of the discriminant, we
can recover the $\QQ_p$-algebra $K$ and the point of $\PP^1(K)$ by taking the
ring $K_f$ and element $[1:\theta]$ of $\PP^1_{K_f}$, up to isomorphism, and thus
the inverse image indeed forms a single orbit.

To check the determinant statement, we observe that this determinant is
preserved after passing from $\QQ_p$ to any unramified extension $F$ of
$\QQ_p$. Furthermore the determinant, and the claimed formula for it, are both
invariant under the action of $\GL_2(\OO_K)$. Using this action, we may
transform $f$ into a form with invertible $x^n$-coefficient and thereby reduce
to the monic case.

The statement in the monic case is that for $F$ a $p$-adic field and $K$ a rank
$n$ étale $F$-algebra, the $p$-adic absolute value of the determinant of the
Jacobian of the map $\alpha \mapsto \operatorname{Nm}_K^F(x-\alpha)$ from
$\OO_K$ to the space of degree $n$ monic polynomials over $\OO_F$, defined
using the $\OO_F$-lattices $\OO_K$ and $\OO_F^n$,
is
\[\abs{\Disc(K/F)}_p^{1/2}\abs{ \Disc(\operatorname{Nm}_K^F(x-\alpha ))}_p^{1/2}.\]
This was essentially proven by Serre \cite[Lemme~2]{Serre78}, but we repeat the
argument here.

We first reduce to the case $K =F^n$. To do this, we base change from $F$ to
the Galois closure $L$ of $K$. observe that if we base change everything from
$F$ to $L$, the map $\alpha \to \operatorname{Nm}_K^F(x-\alpha)$ and its derivative are
preserved, the base change of the $\OO_F$-lattice $\OO_F^n$ is the lattices
$\OO_L^n$, but the base change of the lattice $\OO_K$ is
$\OO_K \otimes_{\OO_F} \OO_L$ which is not $\OO_L^n$. The change of basis matrix to
convert from a basis to $\OO_K$ over $\OO_F$ to the standard basis of $\OO_L^n$
over $\OO_L$ has the property that multiplying by its transpose produces the
matrix of traces whose determinant is $\Disc(K/F)$, and hence its determinant
must be a square root of $\Disc(K/F)$, thus changing from one lattice to
another cancels the $\abs{\Disc(K/F)}_p^{1/2}$ factor and we may reduce to the
case where $K =F^n$.

In this case, the norm map sends a tuple $\alpha_1,\dots, \alpha_n$ to the polynomial
$(x-\alpha_1)\dots (x-\alpha_n)$. The Jacobian is the $n\times n$ matrix whose
$ij$-th entry is $(-1)^{j-1}$ times the $j-1$st elementary symmetric polynomial
of the list $\alpha_1,\dots,\alpha_n$ with $\alpha_i$ removed. The determinant of this matrix
is classical, and known to be
\[\pm \prod_{1\leq i<j \leq n} (\alpha_i-\alpha_j), \] which is the square root of
$\Disc(\operatorname{Nm}(x-\alpha ))$, giving the monic case with $K=F^n$ and
completing the verification of the claim.
\end{proof}

\begin{lemma}\label{p-adic-measure} The measure of the set of
  $[f] \in \PP^n(\QQ_p)$ with $\abs{\Disc(f)}_p \leq p^{-k}$ is
  $ O_n( p^{-k/2})$
  with the implicit constant depending
  only on $n$.
\end{lemma}

\begin{proof} We apply Lemma \ref{change-of-variables} to the function $\psi$
  which is $1$ if $\abs{\Disc(f)}_p \leq p^{-k}$ and $0$ otherwise. For this
  $\psi$, we have
  $\abs{ \Disc(\operatorname{Nm}(P))}_p^{1/2} \psi( \operatorname{Nm}(P)) \leq
  p^{-k/2}$ for all $P$, and thus Lemma \ref{change-of-variables} gives
\[ \int_ {\PP^n(\QQ_p)} \psi d \mu_{\PP^n(\QQ_p)} = \sum_{\substack{K/\QQ_p\\ \deg n \\ \text{\'etale}}}\frac{ \abs{\Disc(K/\QQ_p)}_p^{1/2}}{\abs{\Aut(K)}}  \int_{\PP^1(K)} \abs{ \Disc(\operatorname{Nm}(P))}_p^{1/2} \psi( \operatorname{Nm}(P)) d\mu_{\PP^1(K)}(P)\]
\[ \leq \sum_{\substack{K/\QQ_p\\ \deg n \\ \text{\'etale}}}\frac{ \abs{\Disc(K/\QQ_p)}_p^{1/2}}{\abs{\Aut(K)}}  \int_{\PP^1(K)} p^{-k/2} d\mu_{\PP^1(K)} \leq \sum_{\substack{K/\QQ_p\\ \deg n \\ \text{\'etale}}}\frac{ \abs{\Disc(K/\QQ_p)}_p^{1/2}}{\abs{\Aut(K)}} \cdot p^{-k/2} \cdot \mu_{\OO_K}(\OO_K^\times)^{-1} \]
since the $p$-adic measure of $\PP^1(K)$ is bounded by $\mu_{\OO_K^{m+1}}(\OO_K^{m+1})/\mu_{\OO_K}(\OO_K^\times) = \mu(\OO_K^\times)^{-1}$.

Now $\abs{\Disc(K/\QQ_p)}_p^{1/2}\leq 1$,  $\abs{\Aut(K)}\geq 1$,
$\mu(\OO_K^\times)\geq 2^{-n}$,
and the number of \'{e}tale algebras of degree $n$ over $\QQ_p$ is bounded only in terms of $n$, so
$$\sum_{K/\QQ_p}\frac{ \abs{\Disc(K/\QQ_p)}_p^{1/2}}{\abs{\Aut(K)}\mu(\OO_K^\times)} = O_n(1)$$
and hence the whole expression is $O_n(p^{-k/2})$.  \end{proof}

\begin{lemma}\label{lem:density-of-divisible}
  The fraction of residue classes $f \in V_n (\mathbf Z/p^k\mathbf Z)$ with
  $p^k \mid \Disc(f)$ is $O(p^{-k/2})$ where the implicit constant depends only on
  $n$.
\end{lemma}

\begin{proof}
  The desired fraction of residue classes is the measure of the set
\[
\Omega_k \coloneq \{f=(a_0,\dots,a_n)\in\ZZ_p^{n+1} : p^k \mid \Disc(f)\}.
\]
We partition $\ZZ_p^{n+1}$ according to the greatest common divisor of the
coordinates as $\ZZ_p^{n+1} = \{0\}\sqcup\bigsqcup_{\ell\geq0} p^\ell T$, where
\[
T \coloneq \{(a_0,\dots,a_n)\in\ZZ_p^{n+1} \mid a_0,\dots,a_n\textnormal{ generate the unit ideal of }\ZZ_p\}.
\]
Since $\Disc(p^\ell f) = p^{(2n-2)\ell}\Disc(f)$, we have
\[\Omega_k\cap p^\ell T = p^\ell\cdot\left(\Omega_{\max(0,k-(2n-2)\ell)}\cap T\right)\] for any
$k,\ell\geq0$. By \Cref{p-adic-measure} and the definition of our measure on
$\PP^n(\ZZ_p)$, the set $\Omega_k\cap T$ has measure $O(p^{-k/2})$ for any
$k\geq0$. Hence,
\begin{align*}
    \mu(\Omega_k)
    &= \sum_{\ell\geq0} \mu_{\ZZ_p^{n+1}}\left(\Omega_k\cap p^\ell T\right)
    = \sum_{\ell\geq0} p^{-(n+1)\ell}\cdot \mu_{\ZZ_p^{n+1}}\left(\Omega_{\max(0,k-(2n-2)\ell)}\cap T\right) \\
    &= \sum_{\ell\geq0} p^{-(n+1)\ell}\cdot O\left(p^{\min(0,(n-1)\ell-k/2)}\right)
    = \sum_{\ell\geq0} O(p^{-2\ell - k/2})
    = O(p^{-k/2}).
\qedhere
\end{align*}
\end{proof}

\subsection{Forms whose discriminant is divisible by a large square}\label{ss-large-square}

\begin{lemma}
\label{lem:usually-almost-squarefree}
We have $N(C(s),X) \leq \gamma(s) X^{n+1} + o_s(X^{n+1})$ for $X\to\infty$ for some
$\gamma(s)$ with $\lim_{s\to\infty}\gamma(s)=0$.
\end{lemma}

If we only took into account squarefree integers $a$ in the definition of
$C(s)$, this lemma would follow almost immediately from
\cite[Theorem~5]{BhargavaShankarWang2-2022}.

\begin{proof}[Proof of \Cref{lem:usually-almost-squarefree}]
  Let $M\geq s$. If $\Disc(f)$ is not divisible by $p^2$ for any prime number
  $p>M$, but is divisible by $a^2$ for some integer $a>M$, then it is divisible
  by $a^2$ for some integer $M<a\leq M^2$. (Let $a>M$ be minimal with
  $a^2\mid\Disc(f)$, and assume $a>M^2$. By assumption, the integer $a$ is
  composite. Let $p$ be its smallest prime factor. Then,
  $\frac ap\geq\sqrt a > M$ and $(\frac ap)^2\mid\Disc(f)$, contradicting the
  minimality of $a$.) Hence,
\[
  N(C(s),X) \leq E_1 + E_2
\]
with
\[
  E_1 \coloneq \#\{f\in V_n(\ZZ):\Ht(f)\leq X\textnormal{ and }p^2\mid\Disc(f)\textnormal{ for
    some prime number }p>M\}
\]
and
\[
  E_2 \coloneq \#\{f\in V_n(\ZZ):\Ht(f)\leq X\textnormal{ and }a^2\mid\Disc(f)\textnormal{ for
    some integer }s<a\leq M^2\}.
\]
Let $M = X^{1/(22n^5)}$ and assume that $X$ is large enough so that $M\geq s$ as
above. By \cite[Theorem~5]{BhargavaShankarWang2-2022}, we then have
\[
  E_1 = o(X^{n+1}) \qquad\textnormal{for }X\rightarrow\infty.
\]
For all integers $a\leq M^2$ and all $X\geq1$, we have $a^2\leq M^4\leq X$ and therefore
\[
  \#\{f\in V_n(\ZZ) : \Ht(f)\leq X\textnormal{ and }a^2\mid\Disc(f)\} \ll \omega(a) X^{n+1},
\]
where $\omega(a)$ is the fraction of residue classes $f\in V_n(\ZZ/a^2\ZZ)$ with
$a^2\mid\Disc(f)$. Thus,
\[
E_2 \ll \sum_{a>s} \omega(a) X^{n+1}.
\]
By \Cref{lem:density-of-divisible}, we have $\omega(p^k) = O(p^{-k})$. By
\cite[Proposition~A.1]{BhargavaShankarWang2-2022}, we have
$\omega(p^k)= O(p^{-2})$ for all $k\geq1$. Thus, the Euler product
\[
  \sum_{a\geq1}\omega(a) = \prod_p \sum_{k\geq0} \omega(p^k) = \prod_p \left(1 + \sum_{k\geq1}
    O(p^{-\max(2,k)})\right) = \prod_p (1 + O(p^{-2}))
\]
converges, so $\sum_{a>s}\omega(a)$ goes to $0$ as $s$ goes to infinity.
\end{proof}

\section{The \texorpdfstring{$S_n$}{Sn}-closure}\label{sec:basis-of-closure}

In this section, we construct and prove some elementary properties of the
$S_n$-closure of the rings $R_f$, when $f\in V_n(\ZZ)$ is an element whose
coefficients have greatest common divisor 1.

\subsection{The work of Bhargava--Satriano}

Let $A/B$ be a rank $n$ extension of rings, i.e., $A$ is a $B$-algebra which is
free of rank $n$ as a $B$-module. Then Bhargava--Satriano
\cite{BhargavaSatriano-2014} define the \cdef{$S_n$-closure of $A$ over $B$} to
be the quotient ring $\SnC{A}{B}=A^{\otimes n}/I(A/B)$, where $I(A/B)$ is the ideal
generated by all elements of the form
\begin{equation*}
s_j(a)-\sum_{1\leq i_1<\cdots<i_j\leq n}a^{(i_1)}\cdots a^{(i_j)}
\end{equation*}
for all $a\in A$ and $j\in\{1,\ldots,n\}$. Above, $s_j(a)$ denotes $(-1)^j$ times the
$(n-j)$-th coefficient of the characteristic polynomial of the $B$-module
homomorphism $A\to A$ given by $x\mapsto ax$, and $a^{(i)}$ denotes the tensor
$1\otimes\cdots\otimes a\otimes\cdots\otimes 1$ in $A^{\otimes n}$, with $a$ in the $i$-th position and ones elsewhere.

A number of properties of the $S_n$-closure are proven in
\cite{BhargavaSatriano-2014}, including many we will need. In this subsection,
we recall their construction in \cite[\S6]{BhargavaSatriano-2014} of the
$S_n$-closure of monogenic rings. Let
\[
f(x) = x^n + f_1x^{n-1} + \dots + f_n
\]
be a monic degree $n$ polynomial with coefficients in $B$. Then $R_f=B[x]/f(x)$
is a \cdef{monogenic ring}. The $S_n$-closure of $R_f$ is constructed as
follows: consider the ring $B[X_1,\ldots,X_n]$, and define the element
$\Sigma_k $, for $k\in\{1,\ldots,n\}$, to be the $k$-th elementary symmetric polynomial in
the $X_i$. Then the $S_n$-closure of $R_f$ is
\begin{equation*}
\SnC{R_f}{B} = B[X_1,\ldots,X_n]/(\Sigma_1 +f_1,\Sigma_2 -f_2,\ldots,\Sigma_n -(-1)^nf_n),
\end{equation*}
the quotient of $B[X_1,\dots,X_n]$ by the ideal generated by the relations
$\Sigma_k -(-1)^k f_k$ for $k$ from $1$ to~$n$. They next observed that this
expresses the $S_n$-closure as the tensor product of $B[X_1,\dots,X_n]$ over
$B[\Sigma_1 ,\dots, \Sigma_n ]$ with $B$ (using the $B$-algebra homomorphism
$B[\Sigma_1,\dots,\Sigma_n]\to B$ sending $\Sigma_i$ to $(-1)^i f_i$), so a
$B$-basis of the $S_n$-closure may be provided by a basis for
$B[X_1,\dots,X_n]$ as a module over $B[\Sigma_1 ,\dots, \Sigma_n ]$. They found an
explicit basis of monomials of the form $\prod_{i=1}^n X_i^{e_i}$ where the
$e_i$ satisfy $0\leq e_i <i$ to produce an explicit basis of the $S_n$-closure.

Bhargava--Satriano did not consider the successive minima of the $S_n$-closure
in the case $B=\ZZ$, but this basis could be used to give an upper bound for
the successive minima by calculating the length of each basis vector. Combined
with the discriminant formula \cite[(16)]{BhargavaSatriano-2014}, this gives
lower bounds for the successive minima as well.

\subsection{The $S_n$-closure of binary rings} Our goal is to replicate the
construction of Bhargava--Satriano in the previous subsection for the ring
$R_f$ associated to any primitive $n$-ic form $f\in V_n(\ZZ)$. We first define
analogues of the ring $B[X_1,\dots,X_n]$, the subring
$B[\Sigma_1 ,\dots, \Sigma_n ]$, and a basis for the ring as a module over the subring.
For any ring $R$, let $B_n(R)$ be the subring of
$R[\alpha_1,\dots,\alpha_n, \beta_1,\dots,\beta_n]$ consisting of linear combinations of
monomials such that, for each $i,j$, the exponent of $\alpha_i$ plus the exponent of
$\beta_i$ equals the exponent of $\alpha_j$ plus the exponent of $\beta_j$. Then
$B_n(R)$ is a graded ring over $R$, where we define the degree of a monomial to
be the exponent of $\alpha_i$ plus the exponent of $\beta_i$ (for any $i$). For
$k$ from $0$ to $n$, consider the homogeneous degree $1$ element
\begin{equation*}
  a_k = (-1)^k \sum_{ \substack {S \subseteq \{1,\dots,n\} \\ |S|=k}} \prod_{i \not\in S} \alpha_i \prod_{i\in S} \beta_i
\end{equation*}
of $B_n(R)$. Then $a_k$ is the $k$-th coefficient of the binary $n$-ic form
\begin{equation*}
  f(x,y)=\prod_{i=1}^n(\alpha_i x-\beta_i y).
\end{equation*}
The ring $B_n\colonequals B_n(\ZZ)$ is the analogue of $\ZZ[X_1,\ldots,X_n]$ and the
elements $a_k\in B_n$ are the analogues of the elements $(-1)^k\Sigma_k$. We will
prove that $B_n$ is free of rank $n!$ over $\ZZ[a_0,\ldots,a_n]$, with the following
set $T_n$ of monomials forming a basis:

\begin{definition}
  For $d\in\{0,\dots,n-1\}$, let $T_{n}^{d}$ be the set of monomials in $B_n$ of
  degree $d$ such that for all $0\leq k \leq d-1$, there exists $i<j$ where the
  exponent of $\beta_i$ is $k$ and the exponent of $\beta_j$ is $k+1$. Let
  $T_n \colonequals \bigcup_{d=0}^{n-1} T_{n}^{d}$.
\end{definition}

We first check that $T_n$ has size $n!$ using the following bijection:

\begin{lemma}
  The $($degree $d)$ monomials in $T_{n}^{d}$ are in bijection with the
  permutations in $S_n$ with exactly $d$ ascents (i.e., exactly $d$ indices $i$
  such that $\sigma(i) <\sigma(i+1)$).
\end{lemma}
The number of such permutations is called the \cdef{Eulerian number}
$E(n,d) = |T_{n}^{d}|$.
\begin{proof}
  To a monomial $\prod_{i=1}^n \alpha_i^{d-c_i}\beta_i^{c_i}$ in $T_{n}^{d}$, we associate a
  permutation in $S_n$ whose entries consist of, first, the $i$ such that
  $c_i=0$ in descending order, then the $i$ such that $c_i=1$ in descending
  order, and so on. The condition that there is $i<j$ where $c_i =k $ and
  $c_j=k+1$ implies that the least $i$ with $c_i=k$ is less than the greatest
  $j$ with $c_j=k+1$, so there is an ascent between the $i$ with $c_i=k$ and
  the $i$ with $c_i=k+1$. Since the $i$ with $c_i=k$ are placed in descending
  order for all $k$ there are no more ascents, showing that the number of
  ascents is in fact $d$. In the inverse direction, given a permutation with
  $d$ ascents, let $c_i$ be the number of ascents that appear in the
  permutation before $i$ does. It is straightforward to check that these two
  maps are inverses.
\end{proof}

\begin{corollary}\label{euler-identities}
  We have the following identities: \hfill
\begin{enumerate}[label=(\roman*)]
\item The total number of monomials is $|T_n| = \sum_{d=0}^{n-1}|T_{n}^{d}| = n!$.
\item\label{euler-identities-2} The sum of the degrees of the monomials is
  $\sum_{d=0}^{n-1}|T_{n}^{d}| d = (n-1)n!/2$.
\item\label{euler-identities-3} We have
  $\sum_{d=0}^{n-1} |T_{n}^{d}| \binom{ n+e-d}{n} = (e+1)^n$ for all $e\geq0$.
\end{enumerate}
\end{corollary}
\begin{proof}
\hfill
\begin{enumerate}[leftmargin=*,label=(\roman*)]
\item This is clear since every permutation has between $0$ and $n-1$ ascents.
\item The sum of the degrees of the monomials in $T_n$ is the total number of
  ascents in all permutations. By symmetry (composing with the reflection
  $i\mapsto n+1-i$), we see that this is equal to the total number of descents. The
  total number of ascents plus the total number of descents is $(n-1)n!$.
\item Indeed,
  \[ \sum_{d=0}^{n-1} E(n,d) \binom{ n+e-d}{n} =\sum_{d=0}^{n-1} E(n,n-1-d)
    \binom{e+1+d}{n} = \sum_{d=0}^{n-1} E(n,d) \binom{e+1+d}{n} = (e+1)^n \] by a
  change of variables $d \to n-1-d$, the identity $E(n,n-1-d)=E(n,d)$, and
  Worpitzky's identity. \qedhere
\end{enumerate}
\end{proof}

\begin{lemma}\label{graded-basis}
  $B_n$ is free of rank $n!$ as a module over $\ZZ[a_0,\dots, a_n]$, with basis
  $T_n$.
\end{lemma}

\begin{proof}
  Let us first check that the $n!$ monomials in $T_n$ generate $B_n$. We can
  check that every monomial
  $m = \prod_{i=1}^n \beta_i^{e_i} \alpha_i^{d-e_i} \in B_n$ lies in the submodule
  $\langle T_n\rangle$ generated by $T_n$ by strong induction on the ordering where we
  order lexicographically first by $\deg(m) = d$, then by
  $r_1(m) \colonequals \sum_{i=1}^n e_i^2$, then by
  $r_2(m) \colonequals \sum_{i=1}^n (n-i) e_i$.

  In the case where for each $0\leq k\leq d-1$ we have some $i<j$ with $e_i=k$ and
  $e_{j} =k+1$, we must have $0\leq d\leq n-1$ and so the monomial lies in
  $T_n$. Otherwise, fix some $k$ where this is not the case (so $i>j$ for all
  $i,j$ with $e_i = k$ and $e_j = k+1$) and let $S_0$ be the set of $i$ with
  $e_i\geq k+1$. The monomial
\[
  m' \colonequals \prod_{\substack{1\leq i\leq n:\\ i\notin S_0}} \alpha_i^{d-e_i-1} \beta_i^{e_i}
  \prod_{\substack{1\leq i\leq n:\\ i\in S_0}} \alpha_i^{d-e_i} \beta_i^{e_i-1} \in B_n
\]
lies in $\langle T_n\rangle$ by the induction hypothesis since
$\deg(m') = d-1 < d = \deg(m)$. Hence so does $(-1)^k a_{|S_0|} m'$. This
product is the sum of the monomials
\[
  m''_S \colonequals \prod_{\substack{i\in S\setminus S_0}} \alpha_i^{d-e_i-1} \beta_i^{e_i+1}
  \prod_{\substack{i\in S_0\setminus S}} \alpha_i^{d-e_i+1} \beta_i^{e_i-1}
  \prod_{\substack{\textnormal{all other }i}} \alpha_i^{d-e_i} \beta_i^{e_i} \in B_n
\]
over all subsets $S$ of $\{1,\dots,n\}$ of size $|S| = |S_0|$. The monomial
corresponding to $S = S_0$ is
$\prod_{i=1}^n \beta_i^{e_i} \alpha_i^{d-e_i} = m$. We will check that the other monomials
(corresponding to $S\neq S_0$) lie in $\langle T_n\rangle$ by the induction hypothesis, which
then implies that $m$ lies in $\langle T_n\rangle$.

We have $\deg(m_S'') = \deg(m_S)$. Moreover,
\[
  r_1(m) - r_1(m_S'') = \sum_{i\in S\setminus S_0} \underbrace{\left(e_i^2 - (e_i +
      1)^2\right)}_{= -2e_i-1 \geq -2k-1} + \sum_{i\in S_0\setminus S} \underbrace{\left(e_i^2
      - (e_i - 1)^2\right)}_{= 2e_i-1 \geq 2k+1} \geq 0
\]
since $e_i \leq k$ for all $i\notin S_0$ and $e_i\geq k+1$ for all $i\in S_0$ and
$|S\setminus S_0| = |S_0\setminus S|$. Equality holds if and only if $e_i = k$ for all
$i\in S\setminus S_0$ and $e_i = k+1$ for all $i\in S_0\setminus S$. In the case of equality,
\[
  r_2(m) - r_2(m_S'') = \sum_{i\in S\setminus S_0} (n-i)\left(e_i - (e_i+1)\right) + \sum_{i\in
    S_0\setminus S} (n-i)\left(e_i - (e_i-1)\right) = \sum_{i\in S\setminus S_0} i - \sum_{i\in S_0\setminus S} i
  > 0
\]
since $|S\setminus S_0| = |S_0\setminus S| > 0$ and $i > j$ for all $i,j$ with
$i\in S\setminus S_0$ (and thus $e_i=k$) and $j\in S_0\setminus S$ (and thus
$e_j = k+1$). We have therefore shown that for every $S\neq S_0$ with
$|S| = |S_0|$, the tuple $(\deg(m_S''),r_1(m_S''),r_2(m_S''))$ is
lexicographically smaller than the tuple $(\deg(m),r_1(m),r_2(m))$, so $m_S''$
indeed lies in $\langle T_n\rangle$ by the induction hypothesis. This concludes the proof
that the monomials in $T_n$ generate $B_n$.

The rank of the degree $d$ part of $\ZZ[a_0,\dots,a_n]$ is $\binom{n+d}{n}$.
Hence the rank of the degree $e$ part of the free graded module on the putative
basis is $\sum_{d=0}^{n-1} |T_{n}^{d}| \binom{ n+e-d}{n}$, which equals
$(e+1)^n$ by \Cref{euler-identities}\ref{euler-identities-3}. The rank of the
degree $e$ part of $B_n$ is also $(e+1)^n$. Since the map from the free graded
module on the putative basis to $B_n$ is a surjective map between modules of
the same rank, it is injective, and so the putative basis is in fact a
basis. \end{proof}

Let $R$ be a ring. Pick a \cdef{primitive} element
$f(x,y)=f_0x^n+\cdots+f_ny^n$ in $V_n(R)$, by which we mean an element whose
coefficients generate the unit ideal in $R$. We define the ring $\mathcal{S}(f;R)$ to be
\begin{equation*}
  \mathcal{S}(f;R)\colonequals B_n(R)/(a_0-f_0,a_1-f_1,\ldots,a_n-f_n).
\end{equation*}
Note that $\mathcal{S}(f;R)$ is a free $R$-module of rank $n!$ with basis $T_n$.

We will show in several steps that the $S_n$-closure of $R_f$ is isomorphic to
$\mathcal{S}(f;R)$. To do this, we will first define maps
$\mu^{(j)}\colon R_f \to \mathcal{S}(f;R)$ for $j$ from $1$ to $n$, which we will check are
ring homomorphisms, and then define a ring homomorphism
$\mu \colon R_f^{\otimes n} \to \mathcal{S}(f;R)$ by
$\mu(x_1 \otimes \dots \otimes x_n)=\mu^{(1)} (x_1) \cdots \mu^{(n)} (x_n)$. We will check that
$\mu$ is surjective and that its kernel is $I(R_f/R)$, which implies that it
gives an isomorphism from the $S_n$-closure
$\SnC{R_f}{R} = R_f^{\otimes n}/ I(R_f/R)$ to $\mathcal{S}(f;R)$.

We first define $\mu^{(j)}$ as the unique $R$-linear map $R_f \to \mathcal{S}(f;R)$ such that

\begin{equation}\label{eq:muj-def} 
  \displaystyle \mu^{(j)}(\zeta_k) = 1,  \textrm{ if } k=0, \quad \text{ and } \quad \mu^{(j)}(\zeta_k) = (-1)^{k-1}  \sum_{ \substack {j \in S \subseteq \{1,\dots,n\} \\ |S|=k}} \prod_{i \notin S} \alpha_i \prod_{i\in S} \beta_i,  \textrm{ if }k>0.
\end{equation}

Recall the notation introduced in \Cref{eq:zeta_k}. In the case where $f_0$ is
a unit in $R$, everything simplifies. First when $f_0$ is a unit, $R_f$ is
isomorphic to $R[x]/f(x,1)$ under the map sending $\zeta_0$ to $1$ and
$\zeta_k$ to $f_0 \theta^k+ \dots + f_{k-1} \theta$. Second, when $f_0$ is a unit,
$\mathcal{S}(f;R)$ is a quotient of
$B_n(R) [ a_0^{-1}]= R[ \frac{\beta_1}{\alpha_1},\dots ,\frac{\beta_n}{\alpha_n}, a_0,a_0^{-1}]$
(where
$\frac{\beta_j}{\alpha_j} = \frac{ \alpha_1 \dots \alpha_{j-1} \beta_j \alpha_{j+1} \dots \alpha_n }{a_0} $ with
the numerator in $B_n$) so that
\[\mathcal{S}(f;R) = \left. R\left[ \tfrac{\beta_1}{\alpha_1},\dots ,\tfrac{\beta_n}{\alpha_n}\right]
    \middle/ \left( (-1)^k \Sigma_k \left( \tfrac{\beta_1}{\alpha_1},\dots,
        \tfrac{\beta_n}{\alpha_n}\right) - \tfrac{f_k}{f_0} \right)_{k \in \{1,\dots,n\}
    }\right. \] where $\Sigma_k$ denotes the $k$-th elementary symmetric polynomial.
Note that this is exactly the Bhargava-Satriano formula for the $S_n$-closure
of the monogenic ring $R_f = R_{f/f_0}$.

Our strategy will be to reduce to the case when $f_0$ is a unit in every step
of the argument where this is possible. The key lemma is the following.

\begin{lemma}\label{monogenic-case} When $f_0$ is a unit in $R$, the map
  $\mu^{(j)}$ is under the above isomorphisms the unique ring homomorphism:
  \[ R[x]/f(x,1) \to \left. R\left[ \tfrac{\beta_1}{\alpha_1},\dots
        ,\tfrac{\beta_n}{\alpha_n}\right] \middle/ \left( (-1)^k \Sigma_k \left(
          \tfrac{\beta_1}{\alpha_1},\dots, \tfrac{\beta_n}{\alpha_n}\right) - \tfrac{f_k}{f_0}
      \right)_{k \in \{1,\dots,n\} }\right. \] fixing $R$ and sending
  $\theta \colonequals x \md (f(x,1))$ to $\frac{\beta_j}{\alpha_j}$. \end{lemma}

\begin{proof} Such a ring homomorphism exists since the relations
  $(-1)^k \Sigma_k ( \frac{\beta_1}{\alpha_1},\dots, \frac{\beta_n}{\alpha_n}) =\frac{f_k}{f_0} $
  force $ \prod_{i=1}^n (T - \frac{\beta_i}{\alpha_i}) = f(x,1)/f_0$ so
  $\frac{\beta_j}{\alpha_j}$ satisfies the polynomial equation $f(x,1)=0$. It is unique
  since $\theta$ generates $R[x]/f(x,1)$.

  To check it agrees with $\mu^{(j)}$, it suffices to check on the basis
  $\zeta_0,\dots,\zeta_{n-1}$. For $\zeta_0=1$ this is obvious. For $\zeta_k$ we take the
  formula $f_0 \theta^k+ \dots + f_{k-1} \theta$ for $\zeta_k$ and plug in
  $\frac{\beta_j}{\alpha_j}$ for $\theta$ and $a_{i}$ for $f_i$, and check this agrees with
  $\mu^{(j)}(\zeta_k)$. For $\zeta_1$ this is an easy check. For larger $k$ this follows
  by induction using the identities $\zeta_{k}=\theta(\zeta_{k-1}+f_{k-1})$ and
  $\mu^{(j)}(\zeta_k)=\frac{\beta_j}{\alpha_j}( \mu^{(j)}(\zeta_{k-1})+a_{k-1})$, clear from the
  definition of $\zeta_k$ and the definition \eqref{eq:muj-def} of $\mu^{(j)}$.
\end{proof}

\begin{lemma} $\mu^{(j)}$ is a ring homomorphism. \end{lemma}

\begin{proof} The multiplication of $R_f$ in the basis
  $\zeta_0,\dots, \zeta_{n-1}$ is given by universal polynomials in
  $f_0,\dots, f_n$ (see \cite[Proposition 1.1]{Nakagawa89}) and the same is
  true for the multipication of $\mathcal{S}(f;R)$ in the basis $T_n$, so this statement
  can be expressed as a collection of polynomial identities in $f_0,\dots,f_n$.
  By Lemma \ref{monogenic-case} these identities hold when $f_0$ is invertible
  and hence they hold in general. \end{proof}

The group $\GL_2(R)$ acts on $R[\alpha_1,\dots,\alpha_n,\beta_1,\dots,\beta_n]$ via
$R$-algebra automorphisms, where
$\gamma=\left(\begin{smallmatrix}a&b\\c&d\end{smallmatrix}\right)\in \GL_2(R)$
corresponds to the automorphism $\iota_\gamma$ given by sending
$(\alpha_i,\beta_i)$ to $(a\alpha_i-c\beta_i,-b\alpha_i+d\beta_i)$. This action restricts to an action of
the subalgebra $B_n(R)$ of $R[\alpha_1,\dots,\alpha_n,\beta_1,\dots,\beta_n]$.

We will also denote $\iota_\gamma(t)$ by $t^\gamma$ for any
$t\in R[\alpha_1,\dots,\alpha_n,\beta_1,\dots,\beta_n]$. It is easy to check that for
$f(x,y)\colonequals\prod_i(\alpha_ix-\beta_iy) \in V_n(B_n(R))$, we have
\begin{equation*}
f^\gamma(x,y)=\prod(\alpha_i^\gamma x-\beta_i^\gamma y),
\end{equation*}
where (as always) $f^\gamma$ is defined as in (\ref{eq:action-on-forms}). Note that
$\iota_\gamma^{-1}$ induces an isomorphism
$\mathcal{S}(f;R) \stackrel\sim\rightarrow \mathcal S(f^\gamma;R)$. Indeed, for
$f\in V_n(R)$, let $F_f(x,y)$ denote the binary $n$-ic form
$f(x,y)-\prod_i(\alpha_ix-\beta_iy) \in V_n(B_n(R))$. Then we have
\begin{align*}
  \mathcal{S}(f;R)
  &= B_n(R)/(f_i-a_i)_{i\in\{1,\dots,n\}}
  = B_n(R)/((F_f)_i)_{i\in\{1,\dots,n\}}
  = B_n(R)/((F_f^\gamma)_i)_{i\in\{1,\dots,n\}} \\
  &= B_n(R)/((f^\gamma(x,y) - \prod(\alpha_j^\gamma x-\beta_j^\gamma y))_i)_{i\in\{1,\dots,n\}} \\
  &\cong B_n(R)/((f^\gamma(x,y) - \prod(\alpha_j x-\beta_j y))_i)_{i\in\{1,\dots,n\}}
  = B_n(R)/(f^\gamma_i - a_i)_{i\in\{1,\dots,n\}}
  = \mathcal S(f^\gamma;R),
\end{align*}
where the isomorphism on the third line is induced by $\iota_\gamma^{-1}$ (which acts
trivially on $f^\gamma$).

We will be using Nakagawa's explicit isomorphism $\psi_\gamma$ between the rings
$R_f$ and $R_{f^\gamma}$ in the next result. Note however that Nakagawa uses the
right action of $\GL_2$ on $V_n$ while we use the left action. To translate
between these actions it is only necessary to take transposes of elements in
$\GL_2(R)$. That is, Nakagawa's $\gamma\cdot f$ is our $f^{\gamma^t}$.
\begin{lemma}\label{lem-change-basis-cd} Let $f \in V_n(R)$ be an element whose
  coefficients generate the unit ideal, and let $\gamma \in \mathrm{GL}_2(R)$. We
  obtain the commutative diagram
\[ \begin{tikzcd}
R_f \arrow[d, "\psi_\gamma"] \arrow[r,"\mu^{(j)}"] & \mathcal{S}(f;R) \arrow[d,"\iota_\gamma^{-1}"] \\
R_{f^\gamma } \arrow[r,"\mu^{(j)}"] & \mathcal{S}(f^\gamma ;R)
\end{tikzcd} \] where $\psi_\gamma$ is the isomorphism
$R_f \to R_{f^\gamma }=R_{\gamma^t\cdot f}$ from \cite[Corollary on p. 226]{Nakagawa89} and
$\iota_\gamma^{-1}$ is the isomorphism described above. \end{lemma}

\begin{proof} 
  To check the diagram commutes, observe that, expressed in terms of the
  standard bases, all maps are given by universal polynomials in
  $f_0,\dots, f_n$ and the coefficients of $\gamma$, so the commutativity of the
  diagram is equivalent to an identity of polynomials. Hence it suffices to
  check the diagram commutes in the case when $f_0$ and $(f^\gamma)_0$ are
  invertible. In this case, $R_f$ is generated by the element $\theta$, so it
  suffices to check $\theta$ is sent to the same element by both sides of the
  commutative diagram.

  Let $\gamma = \begin{psmallmatrix} a & b \\ c & d \end{psmallmatrix}$. The
  horizontal arrow sends $\theta$ to $\frac{\beta_j}{\alpha_j}$ and by definition
  $(\iota_\gamma)^{-1}$ sends $\beta_j$ to
  $(ad-bc)^{-1} (b \alpha_j + a \beta_j)$ and $\alpha_j$ to
  $(ad-bc)^{-1} ( d \alpha_j + c\beta_j)$ so $\iota_\gamma^{-1} $ sends
  $\frac{\beta_j}{\alpha_j}$ to
  $\frac{ b \alpha_j + a \beta_j}{d \alpha_j + c\beta_j}$. On the other hand, by definition, the
  vertical arrow $\psi_\gamma$ sends $\theta$ to
  $\frac{ a \theta+ b}{ c\theta+d}$. (The isomorphism $\psi_\gamma$ is constructed as
  $\sigma$ in \cite[Proof of Proposition 1.2]{Nakagawa89} by reduction to three
  different cases. However, all three cases are applicable when $f$ has Galois
  group $S_n$, in which case there is a unique field isomorphism
  $K_f \to K_{f^\gamma}$ which is given by the formula
  $\theta \mapsto \frac{ a \theta+ b}{ c\theta+d}$, so all three must be given by this formula, and
  thus the general case is given by this formula as well.) This is then sent by
  $\mu^{(j)}$ to
  $(a \frac{\beta_j}{\alpha_j} + b)/( c\frac{\beta_j}{\alpha_j}+d)$. Since these are equal, the
  diagram commutes. \end{proof}

We now define $\mu$ as the unique ring homomorphism
$R_f^{ \otimes n} \to \mathcal{S}(f;R) $ such that
$\mu(x_1 \otimes \dots \otimes x_n)=\mu^{(1)} (x_1) \cdots \mu^{(n)} (x_n)$ for
$x_1,\dots, x_n\in R_f$. It is immediate from the definition of the tensor
product of rings that $\mu$ exists, is unique, and is a ring homomorphism.

It also follows from Lemma \ref{lem-change-basis-cd} that the diagram

\begin{equation}\label{eq-change-basis-cd} \begin{tikzcd} R_f^{\otimes n} \arrow[d,
    "\psi_\gamma^{\otimes n} "] \arrow[r,"\mu"] & \mathcal{S}(f;R) \arrow[d,"\iota_\gamma^{-1}"] \\ R_{f^\gamma }^{\otimes n}
    \arrow[r,"\mu"] & \mathcal{S}(f^\gamma ;R)\end{tikzcd} \end{equation} commutes.

\begin{lemma}\label{monogenic-surjective} When $f_0$ is invertible, $\mu$ is
  surjective.\end{lemma}

\begin{proof} This follows from Lemma \ref{monogenic-case} since we saw that
  $\mathcal{S}(f;R)$ is generated by
  $\frac{\beta_1}{\alpha_1},\dots, \frac{\beta_n}{\alpha_n}$ in this case and
  $\mu^{(j)}$ sends $\theta$ to $\frac{\beta_j}{\alpha_j}$ so
  $\frac{\beta_j}{\alpha_j}$ is in the image of $\mu$.\end{proof}

\begin{lemma}\label{monogenic-kernel} When $f_0$ is invertible, the kernel of
  $\mu$ is the ideal $I(R_f,R)$. \end{lemma}

\begin{proof} The surjection $R[x] \to R[x]/f(x,1) \cong R_f$ gives a surjection
  $R[x_1,\dots,x_n] \to R_f^{\otimes n}$.

  By Lemma \ref{monogenic-case}, the composition
  $R[x_1,\dots,x_n] \to R_f^{\otimes n} \stackrel\mu\to \mathcal S(f;R) $ is the map
  \[R[x_1,\dots, x_n] \to \left. R\left[ \tfrac{\beta_1}{\alpha_1},\dots
        ,\tfrac{\beta_n}{\alpha_n}\right] \middle/ \left( (-1)^k \Sigma_k \left(
          \tfrac{\beta_1}{\alpha_1},\dots, \tfrac{\beta_n}{\alpha_n}\right) - \tfrac{f_k}{f_0}
      \right)_{k \in \{1,\dots,n\} }\right)\] sending $x_i$ to
  $\frac{\beta_i}{\alpha_i}$, so the kernel of the composition is
  $ \left( (-1)^k \Sigma_k \left( x_i ,\dots, x_n \right) - \tfrac{f_k}{f_0}
  \right)_{k \in \{1,\dots,n\} } $. Thus the kernel of $\mu$ is generated by the
  images of the classes
  $(-1)^k \Sigma_k \left( x_i ,\dots, x_n \right) - \tfrac{f_k}{f_0}$ in
  $R_f^{\otimes n}$. Since the image of $x_i$ is $x^{(i)}$, the image of the class
  $(-1)^k \Sigma_k \left( x_i ,\dots, x_n \right) - \tfrac{f_k}{f_0}$ is, up to a
  sign, equal to
  $s_j(a)-\sum_{1\leq i_1<\cdots<i_j\leq n}a^{(i_1)}\cdots a^{(i_j)}$ where
  $j=k$ and $a=x$ and thus lies in $I(R_f/R)$ by definition.

  Hence the kernel of $\mu$ is contained in $I(R_f/R)$. The quotient of
  $R_f^{\otimes n}$ by the kernel of $\mu$ is $\mathcal S(f;R)$ by Lemma
  \ref{monogenic-surjective}, and we know that $\mathcal S(f;R)$ is free of
  rank $n!$. On the other hand, the quotient of $R_f^{\otimes n}$ by $I(R_f/R)$ is
  free of rank $n!$ by \cite[Theorem 3]{BhargavaSatriano-2014}. Hence these
  ideals are equal, since it is not possible for a free module of rank $n!$ to
  be the quotient of a free module of rank $n!$ by a nonzero
  submodule. \end{proof}

\begin{lemma}\label{lem:sn-closure-is-bn} Let $f$ be a primitive element of $V_n(R)$.
  Then the map $\mu \colon R_f^{\otimes n} \to \mathcal{S}(f;R) $ is surjective with kernel
  $I(R_f/R)$, thereby inducing an isomorphism
  $\SnC{R_f}{R}\stackrel\sim\to \mathcal{S}(f;R)$. \end{lemma}
\begin{proof}

  Since $\mu$ is a map of finitely generated $R$-modules, to check $\mu$ is
  surjective, it is enough that $\mu$ is surjective after base changing to
  $\kappa$ for all residue fields $\kappa$ of $R$. To check this, we may further base
  change to any field extension $\kappa'$ of $\kappa$. We can always choose
  $\kappa'$ so that $|\kappa'|>n$, which ensures that $f(x,y)$ does represent a nonzero
  (hence unit) element in $\kappa'$. (Here we are using that $f$ is primitive.) Then
  we can choose $\gamma$ so that $f^\gamma $ has leading coefficient a unit. For this
  $\gamma$, in the commutative diagram \eqref{eq-change-basis-cd}, the columns are
  isomorphisms and by Lemma \ref{monogenic-surjective} the bottom arrow is
  surjective, so the top arrow is surjective, giving the desired surjectivity
  of $\mu$.

  To check that the kernel of $\mu$ is $I(R_f/R)$, we need a slightly more
  complicated version of the same argument. To verify that two ideals are
  equal, it is enough to check that they are equal after base changing to the
  localization at every prime ideal. To check this, we may further base change
  to any finite \'{e}tale algebra over this local ring. We can choose this
  finite \'{e}tale extension so that its residue field has more than $n$
  elements, which ensures that $f(x,y)$ represents a nonzero element in the
  residue field and hence a unit in the ring. We again choose $\gamma$ so that
  $f^\gamma $ has leading coefficient a unit. For this $\gamma$, in the commutative
  diagram \eqref{eq-change-basis-cd}, the columns are isomorphisms, so the
  column isomorphism sends the kernel of the top arrow to the kernel of the
  bottom arrow, which by Lemma \ref{monogenic-case} is $I(R_{f^\gamma}/R)$. The
  definition of $I(R_f/R)$ is manifestly also compatible with the isomorphism
  $R_{f^\gamma}\to R_f$, so we get that the kernel of $\mu$ is $I(R_f/R)$. \end{proof}

\begin{remark} The construction of $\mathcal{S}(f;R)$ has the following geometric
  interpretation.

  We first state a general fact in algebraic geometry: Let $X$ a scheme over
  $R$, $g \colon X \to \PP^n_R$ a finite morphism, $x_0,\dots, x_n$ the standard
  basis for $H^0( \PP^n, \OO(1))$, and $[y_0:\dots: y_n]$ a point of
  $\PP^n(R)$. Then,
  \[\Spec \left[ \bigoplus_{d=0}^\infty H^0(X, g^* \OO_{\mathbb P^n}(d)) /
      (x_0-y_0,\dots, x_n-y_n)\right] \] is the fiber product of the following
  diagram

\begin{equation*}
    \begin{tikzcd}
                & & X \arrow[d, "g"] \\
        \Spec R \arrow[rr, "{[y_0:\dots :y_n]}"'] & & \PP^n \, .
    \end{tikzcd}
\end{equation*}

Taking 
\begin{itemize}
\item $X = (\PP^1)^n_R$ where the $i$th copy of $\PP^1$ has coordinates
  $(\beta_i,\alpha_i)$,
\item $g \colon X \to \PP^n$ the map sending
  $([\beta_1:\alpha_1],\dots,[\beta_n,\alpha_n])$ to $[a_0,\dots,a_n]$, and
\item $[y_0,\dots,y_n]=[f_0,\dots,f_n]$,
\end{itemize}
we see that $\Spec \mathcal{S}(f;R)$ is the inverse image of $[f_0,\dots,f_n]$ under
$g$. Indeed, $g^* \OO_{\PP^n} (1) = O_{ (\PP^1)^n}(1,\dots, 1)$ so that
$ \bigoplus_{d=0}^\infty H^0(X, g^* \OO_{\PP^n}(d))= B_n(R)$ and then
$(x_0-y_0,\dots, x_n-y_n) = (a_0-f_0,\dots,a_n-f_n)$.

So $\Spec \mathcal{S}(f;R)$ is a closed subscheme of $(\mathbb P^1)^n_R$. Classically,
$\Spec R_f$ is the closed subscheme of $\PP^1_R$ defined by the binary form
$f$. The ring homomorphism $\mu^{(j)} \colon R_f\to \mathcal{S}(f;R)$ corresponds to a map
$\Spec \mathcal{S}(f;R)\to \Spec R_f$ arising from the projection onto the
$j$th copy of $\PP^1$.

It is likely possible to prove Lemma \ref{lem:sn-closure-is-bn} from this
geometric description, but we have chosen to express the proof in terms of
elementary algebra instead.\end{remark}

We end this section with a determination of the discriminant polynomial of the
$S_n$-closure of $R_f$.
\begin{lemma}\label{disc-of-closure}The discriminant of the $S_n$-closure
  $\mathcal{S}(f;R) \cong\SnC{R_f}{R}$ of $R_f$ with respect to the basis $T_n$ is
  $\Disc(f)^{n!/2}$. \end{lemma}

\begin{proof} 
  The discriminant must be a polynomial in the coefficients of $f$. Denote this
  polynomial by $\Disc_G(f)$. Suppose $f$ is a polynomial over $\ZZ$ with
  squarefree discriminant that represents $ 1$. Then $R_f$ is a monogenic ring
  so we may apply \cite[(16)]{BhargavaSatriano-2014} to deduce that the
  $S_n$-closure of $R_f$ has discriminant $\Disc(f)^{n!/2}$. (While the
  discriminant of an arbitrary ring extension depends on a choice of basis, and
  our basis may be different from Bhargava and Satriano's, when the base ring
  is $\ZZ$ the discriminant is independent of the choice of basis.) The set of
  polynomials in $V_n(\ZZ)$ with squarefree discriminant that represent $1$ is
  Zariski dense (in $V_{n,\QQ}$) and so we have $\Disc(f)^{n!/2}= \Disc_G(f)$
  as polynomials in the coefficients of $f$. In particular, they are equal for
  $f$ defined over an arbitrary ring. \end{proof}

\section{Successive minima}
\label{sec:successive-minima}

In this section, we derive estimates for the successive minima of the
$S_n$-closures of $R_f$, for $f\in V_n(\ZZ)$. Throughout this section,
$f=[f_0,\dots,f_n]\in V_n(\ZZ)$ will denote a primitive irreducible binary
$n$-ic form. In particular, $f_0\neq0$. In \Cref{lem:sn-closure-is-bn}, we have
constructed an isomorphism
\[
    \mathcal S(f;\ZZ) = B_n(\ZZ)/(a_0-f_0,\dots,a_n-f_n) \stackrel\sim\to \SnC{R_f}{\ZZ}.
\]
Let 
\[
    \tau_f : B_n(\ZZ) \twoheadrightarrow \SnC{R_f}{\ZZ}
\]
be the resulting map.

We have the following commutative diagram of $S_n$-equivariant ring
homomorphisms, where the vertical isomorphisms come from
\Cref{lem:sn-closure-is-bn} and the horizontal injections in the bottom row are
induced by the inclusions $R_f \hookrightarrow \OO_{K_f} \hookrightarrow K_f$:
\[
\begin{tikzcd}
\mathcal S(f;\ZZ) \dar{\sim} \ar[hook]{rr} & & \mathcal S(f;\QQ) \dar{\sim} \\
\SnC{R_f}{\ZZ} \rar[hook] & \SnC{\OO_{K_f}}{\ZZ} \rar[hook] & \SnC{K_f}{\QQ}
\end{tikzcd}
\]

For any degree $n$ extension $K/\QQ$, we define an \cdef{absolute value} on
$\SnC{K}{\QQ}$ by
\[
    \|r\| \colonequals \max_{\sigma} {\lvert\sigma(r)\rvert},
\]
where $\sigma$ runs over all $n!$ $\QQ$-algebra homomorphisms $\SnC{K}{\QQ}\to\CC$.

For any $d\in\{0,\dots,n-1\}$, we let $\hatKfm{f}{d}\subseteq\hatKf{f}$ be the
$\QQ$-vector space spanned by the images in $\hatKf{f}$ of the monomials in
$B_n(\ZZ)$ of degree at most $d$, so that we have a filtration
\[
\QQ = \hatKfm{f}{0} \subseteq \hatKfm{f}{1} \subseteq \dots \subseteq \hatKfm{f}{n-1} = \hatKf{f}.
\]

For ``generic'' $f$ this filtration roughly corresponds to a filtration by
successive minima: $\SnC{\OO_{K_f}}{\ZZ}\cap\hatKfm{f}{d}$ is generated by
elements of absolute value $\ll H(f)^d$ and contains all elements of
$\SnC{\OO_{K_f}}{\ZZ}$ of absolute value $\ll H(f)^d$. (See
\Cref{monomial-length-upper-bound,lem:other-generators-long}.)

\begin{lemma}\label{sn-invariants} Let $S_n$ act on $B_n(\mathbf Z)$ by
  permuting the variables $\alpha_1,\dots,\alpha_n$ and simultaneously permuting the
  variables $\beta_1,\dots,\beta_n$. We have
  $B_n(\mathbf Z)^{S_n} = \mathbf Z[a_0,\dots,a_n]$. \end{lemma}

\begin{proof} Certainly $a_0,\dots,a_n$ are $S_n$-invariant so we have
  $\mathbf Z[a_0,\dots,a_n] \subseteq B_n(\mathbf Z)^{S_n}$. By Lemma
  \ref{graded-basis}, $ B_n(\mathbf Z)$ is free as a
  $\mathbf Z[a_0,\dots,a_n]$-module with a basis including $1$, so
  $B_n(\mathbf Z)/\mathbf Z[a_0,\dots,a_n]$ is free as a
  $\mathbf Z[a_0,\dots,a_n]$-module. Then
  $B_n(\mathbf Z)^{S_n} / \mathbf Z[a_0,\dots,a_n] $ is a submodule.

  After inverting $a_0$, $B_n(\mathbf Z)[a_0^{-1}]$ is simply the polynomial
  ring on generators $\frac{\beta_1}{\alpha_1},\dots, \frac{\beta_n}{\alpha_n}$ over the base
  ring $\mathbf Z [a_0,a_0^{-1}]$. By the fundamental theorem of symmetric
  polynomials, the $S_n$-invariants of this ring are generated by the
  elementary symmetric polynomials in
  $\frac{\beta_1}{\alpha_1},\dots, \frac{\beta_n}{\alpha_n}$, which are the elements
  $\frac{a_1}{a_0}, \dots, \frac{a_n}{a_0}$ of
  $\mathbf Z[a_0,\dots,a_n,a_0^{-1}]$.

  So
  \[B_n(\mathbf Z)^{S_n} \subseteq ( B_n(\mathbf Z)[a_0^{-1}])^{S_n} = \mathbf
    Z[a_0,\dots,a_n,a_0^{-1}] \] so
  $B_n(\mathbf Z)^{S_n} /\mathbf Z[a_0,\dots,a_n]$ is $a_0$-power torsion, but
  as a submodule of a free module
  $B_n(\mathbf Z)^{S_n} /\mathbf Z[a_0,\dots,a_n]$ is torsion free, and hence
  it vanishes, as desired.\end{proof}

\begin{lemma}\label{monomial-length-upper-bound}
  There is a constant $C>0$ such that for all monomials
  $m = \prod_{i=1}^n \alpha_i^{e_i} \beta_i^{d-e_i}$ in $B_n(\ZZ)$ of
  degree~$d$ and for all primitive $f\in V_n(\ZZ)$, we have
  $\|\tau_f(m)\| \leq C\cdot\Ht(f)^d$.
\end{lemma}
\begin{proof} By definition it suffices to show for each homomorphism
  $\sigma\colon \SnC{K_f}{\QQ} \to \mathbf C$ that $\abs{ \sigma(\tau_f(m))} \ll \Ht(f)^d$.

  Observe that $m$ is a root of the monic polynomial equation
  \[ \sum_{k=0}^{n!} c_k y^{n!-k}: = \prod_{ g\in S_n} (y - \prod_{i=1}^n \alpha_{g(i)}^{e_{i}}
    \beta_{g(i)}^{d-e_{i}}) . \] Then $c_k $ is an element of $B_n(\mathbf Z)$ of
  degree $k d$ which is $S_n$-invariant, thus by Lemma \ref{sn-invariants} a
  polynomial in $a_0,\dots, a_n$ of degree $kd$. Since $\tau_f$ factors through
  $B_n(\mathbf Z)/(a_0-f_0,\dots,a_n-f_n)$ and therefore sends $a_i$ to $f_i$,
  we have $\abs{\sigma(\tau_f(c_k) ))} \ll \Ht(f)^{dk}$. Then
  $\sigma(\tau_f(m))$ is a root of the polynomial
  $\sum_{k=0}^{n!} \sigma(\tau_f(c_k)) y^{n!-k}$ and thus is $\ll \Ht(f)^d$ by a classical
  bound due to Lagrange. \end{proof}

\begin{remark}
  One can also show a lower bound $\|\tau_f(m)\|\gg\Ht(f)$ for all monomials
  $m$ of degree $1$, but this won't be needed.
\end{remark}

\begin{lemma}\label{not-degree-1-lower-bound}
  Let $d\in\{0,\dots,n-1\}$. For all $\varepsilon>0$ and all integers $s\geq1$, there is a
  number $\varphi(\varepsilon,s)>0$ such that for all primitive
  $f\in V_n(\ZZ)\setminus(B(\varepsilon)\cup C(s))$ (with $B(\varepsilon)$ and
  $C(s)$ defined in \Cref{sec:special-forms}), all $r\in \SnC{\OO_{K_f}}{\ZZ}$
  with $\|r\|\leq\varphi(\varepsilon,s)\Ht(f)^{d+1}$ lie in $\hatKfm{f}{d}$.
\end{lemma}
\begin{proof}
  Let $f\in V_n(\ZZ)$ be primitive. In \Cref{graded-basis}, we have constructed a
  $\ZZ$-basis of $\mathcal S(f;\ZZ)=\SnC{R_f}{\ZZ}$, the $S_n$-closure of
  $R_f$, consisting of $|T_n^d| = E(n,d)$ monomials of degree $d$ for
  $d\in\{0,\dots,n-1\}$. Let $v_1,\dots,v_{n!}$ be this basis, with the monomials
  ordered increasingly by their degree. By the previous lemma, we have
  $\|v_i\|\ll s_i'$ for all $i$, where $s_1'\leq\dots\leq s_{n!}'$ are the following
  numbers:
\[
  \underbrace{H(f)^0\leq\dots\leq
    H(f)^0}_{E(n,0)}\leq\dots\leq\underbrace{H(f)^{n-1}\leq\dots\leq
    H(f)^{n-1}}_{E(n,n-1)}.
\]
Let $s_1,\dots,s_{n!}$ be the successive minima of $\SnC{\OO_{K_f}}{\ZZ}$ (with
respect to $\|\cdot\|$). Since $v_1,\dots,v_{n!}$ are linearly independent elements
of $\mathcal S(f;\ZZ) \subseteq \SnC{\OO_{K_f}}{\ZZ}$, we must have
$s_i \ll s_i'$ for all $i$.

If $f\notin B(\varepsilon)\cup C(s)$, then
$\lvert\Disc(f)\rvert \geq \varepsilon\Ht(f)^{2(n-1)}$ and
$[\OO_{K_f}:\mathcal S(f;\ZZ)]\leq s$. The latter condition implies that
$[\SnC{\OO_{K_f}}{\ZZ} : \mathcal S(f;\ZZ)]\ll_s 1$ and hence
$\Disc(\SnC{\OO_{K_f}}{\ZZ})\gg_s\Disc(\mathcal S(f;\ZZ))$. Then:
\begin{align*}
  s_1'\cdots s_{n!}'
  &= \prod_{d=0}^{n-1} H(f)^{d E(n,d)} && \textnormal{(by definition)} \\
  &= H(f)^{\sum_{d=0}^{n-1} d E(n,d)} \\
  &= H(f)^{(n-1)n!/2} && \textnormal{(by \Cref{euler-identities}\ref{euler-identities-2})}\\
  &\ll_\varepsilon \lvert\Disc(f)\rvert^{n!/4} && \textnormal{(since $f\notin B(\varepsilon)$)} \\
  &= \lvert\Disc(\mathcal S(f;\ZZ))\rvert^{1/2} && \textnormal{(by \Cref{disc-of-closure})} \\
  &\ll_s \lvert\Disc(\SnC{\OO_{K_f}}{\ZZ})\rvert^{1/2} && \textnormal{(since $f\notin C(s)$)} \\
  &\ll s_1\cdots s_{n!} && \textnormal{(by Minkowski's second theorem)}
\end{align*}
Together with the estimates $s_i\ll s_i'$, this implies
$s_i\asymp_{\varepsilon,s} s_i'$ for all $i$.

Now, let $d\in\{0,\dots,n-1\}$ and let
$I \colonequals E(n,0) + \dots + E(n,d)$, the number of monomials of degree at
most~$d$ in our $\ZZ$-basis $v_1,\dots,v_{n!}$ of $\mathcal S(f;\ZZ)$. If
$r\in\SnC{\OO_{K_f}}{\ZZ}\setminus\hatKfm{f}{d}$, then the $I+1$ vectors
$v_1,\dots,v_I \in \hatKfm{f}{d}$ and $r\notin\hatKfm{f}{d}$ are linearly independent.
In particular, one of them must have length
$\geq s_{I+1} \gg_{\varepsilon,s} H(f)^{d+1}$. The vectors $v_1,\dots,v_I$ have length
$\ll H(f)^d$, so for sufficiently large $H(f)$, they cannot have length
$\gg_{\varepsilon,s} H(f)^{d+1}$. Then, $r$ must have length
$\gg_{\varepsilon,s} H(f)^{d+1}$. The general claim follows as there are only finitely many
$f\in V_n(\ZZ)$ with bounded height $H(f)$ (and for each of them, $\|r\|$ is
bounded from below by some positive constant for all
$0\neq r\in\SnC{\OO_{K_f}}{\ZZ}$).
\end{proof}

\begin{lemma}
\label{lem:other-generators-long}
For any constant $C>0$, there is a subset $V_n''(\ZZ)$ of $V_n'(\ZZ)$ of
density $1$ in $V_n'(\ZZ)$ such that for all $d\in\{0,\dots,n-1\}$, all
$f\in V_n''(\ZZ)$, and all $r\in\SnC{\OO_{K_f}}{\ZZ}\setminus\hatKfm{f}{d}$, we have
$\|r\|>C\cdot\Ht(f)^d$.
\end{lemma}
\begin{proof}
  By \Cref{not-degree-1-lower-bound}, for any $\varepsilon>0$ and $s\geq1$, there is some
  ball $D(\varepsilon,s)\subseteq V_n(\RR)$ such that for all
  $f\in V_n'(\ZZ)\setminus(B(\varepsilon)\cup C(s)\cup D(\varepsilon,s))$ and all
  $r\in\SnC{\OO_{K_f}}{\ZZ}\setminus\hatKfm{f}{d}$, we have
  $\|r\|>C\cdot\Ht(f)^d$. We can thus take
\[
  V_n''(\ZZ) \colonequals \bigcup_{\substack{\varepsilon>0,\\ s\geq1}} V_n'(\ZZ)\setminus(B(\varepsilon)\cup C(s)\cup
  D(\varepsilon,s)).
\]
The set $V_n'(\ZZ)\setminus(B(\varepsilon)\cup C(s)\cup D(\varepsilon,s))$ has the same density in
$V_n'(\ZZ)$ as the set $V_n'(\ZZ)\setminus(B(\varepsilon)\cup C(s))$ since the ball
$D(\varepsilon,s)$ only contains finitely many forms $f\in V_n'(\ZZ)$. Thus, the (lower)
density of $V_n''(\ZZ)$ is greater than or equal to $1$ minus the density of
$B(\varepsilon)$ minus the density of $C(s)$. By \Cref{lem:usually-large-disc} and
\Cref{lem:usually-almost-squarefree}, these densities go to $0$ as
$\varepsilon\to0$ and $s\to\infty$.
\end{proof}

\begin{lemma}
\label{lem:deg1-usually-equal}
There is a subset $V_n''(\ZZ)$ of $V_n'(\ZZ)$ of density $1$ such that for all
$f,f'\in V_n''(\ZZ)$ giving rise to the same number field
$K\coloneq K_f=K_{f'}$ and all $d\in\{0,\dots,n-1\}$, we have
$\hatKfm{f}{d}=\hatKfm{f'}{d}$ (as subspaces of $\hatK=\hatKf{f}=\hatKf{f'}$).

More precisely: for any isomorphism $K_f\cong K_{f'}$, the resulting isomorphism
$\hatKf{f}\cong\hatKf{f'}$ sends $\hatKfm{f}{d}$ to $\hatKfm{f'}{d}$.
\end{lemma}
\begin{proof}
Let $C$ as in \Cref{monomial-length-upper-bound} and then let $V_n''(\ZZ)$ as in \Cref{lem:other-generators-long}. Assume without loss of generality that $\Ht(f)\leq \Ht(f')$. The $\QQ$-vector space $\hatKfm{f}{d}$ is according to \Cref{monomial-length-upper-bound} generated by vectors $m\in \mathcal S(f;\ZZ)\subseteq\SnC{\OO_{K_f}}{\ZZ}$ with $\|m\| \leq C\cdot \Ht(f)^d$. By \Cref{lem:other-generators-long}, these vectors must all lie in $\hatKfm{f'}{d}$. Thus, $\hatKfm{f}{d} \subseteq \hatKfm{f'}{d}$. On the other hand, the vector spaces $\hatKfm{f}{d}$ and $\hatKfm{f'}{d}$ have the same dimension, so we in fact have equality.
\end{proof}

\section{Uniqueness of binary forms}
\label{sec:without-representation-theory}

To show that two forms in $V_n(\QQ)$ are $G(\QQ)$-equivalent, we will show that
their roots have the same cross-ratios.

\begin{definition}
  Let $F$ be a field. The \cdef{cross ratio} of four distinct points
  $P_1,\dots,P_4\in\PP^1(F)$ with $P_i=[\beta_i:\alpha_i]$ is the point
\[
(P_1,P_2;P_3,P_4)
\coloneq
\left[
(\alpha_1\beta_3-\beta_1\alpha_3)(\alpha_2\beta_4-\beta_2\alpha_4)
:
(\alpha_1\beta_4-\beta_1\alpha_4)(\alpha_2\beta_3-\beta_2\alpha_3)
\right] \in \PP^1(F).
\]
\end{definition}

\begin{example}
  If $P_1=[1:0]$ and $P_2=[0:1]$ and $P_3=[1:1]$, then $(P_1,P_2;P_3,P_4)=P_4$.
\end{example}

Famously, cross ratios can be used to determine whether two lists of points in
$\PP^1(F)$ are related by a projective transformation:

\begin{lemma}
\label{lem:cross-ratio-map}
Let $n\geq3$ and let $P_1,\dots,P_n$ and $P_1',\dots,P_n'$ be two $n$-tuples of
distinct points in $\PP^1(F)$. There is a matrix $g\in\PGL_2(F)$ sending $P_i$ to
$P_i'$ for all $i=1,\dots,n$ if and only if
$(P_1,P_2;P_3,P_j) = (P_1',P_2';P_3',P_j')$ for all $j=4,\dots,n$.
\end{lemma}
\begin{proof}
  The group $\PGL_2(F)$ acts simply transitively on triples of distinct points
  in $\PP^1(F)$. The cross ratio is invariant under the action of $\PGL_2(F)$
  on $\PP^1(F)$. The claim then follows from the previous example.
\end{proof}

Now, let $n\geq3$. For any $f\in V_n'(\ZZ)$ (which certainly satisfies
$\alpha_1\cdots\alpha_n = f_0 \neq 0$ in the field $\hatKf{f}$), we define some elements
$\alpha_{f,1},\dots,\alpha_{f,n},\beta_{f,1},\dots,\beta_{f,n}$ of $\hatKf{f}$ as follows:
\[
	\alpha_{f,1},\dots,\alpha_{f,n-1}\colonequals1
	\qquad\textnormal{and}\qquad
	\alpha_{f,n} \colonequals f_0,
\]
\[
	\frac{\beta_{f,i}}{\alpha_{f,i}} \colonequals \frac{\alpha_1\dots\alpha_{i-1}\beta_i\alpha_{i+1}\dots\alpha_n}{\alpha_1\dots\alpha_{i-1}\alpha_i\alpha_{i+1}\dots\alpha_n}
	\qquad\textnormal{for }i=1,\dots,n.
\]
In $\hatKf{f}$, any monomial in $\alpha_i,\beta_i$ lying in $B_n(\ZZ)$ is equal to the
corresponding monomial in $\alpha_{f,i},\beta_{f,i}$:
\[
  \prod_{i=1}^n \alpha_i^{d-e_i} \beta_i^{e_i} = (\alpha_1\cdots\alpha_n)^d \prod_{i=1}^n
  \left(\frac{\alpha_1\cdots\alpha_{i-1}\beta_i\alpha_{i+1}\cdots\alpha_n}{\alpha_1\cdots\alpha_{i-1}\alpha_i\alpha_{i+1}\cdots\alpha_n}\right)^{e_i}
  = (\alpha_{f,1}\cdots\alpha_{f,n})^d \prod_{i=1}^n \left(\frac{\beta_{f,i}}{\alpha_{f,i}}\right)^{e_i} =
  \prod_{i=1}^n \alpha_{f,i}^{d-e_i} \beta_{f,i}^{e_i}.
\]
In particular, we have
\begin{equation}\label{eq:concrete-factorization}
  f(x,y) = \prod_{i=1}^n (\alpha_{f,i}x - \beta_{f,i}y),
\end{equation}
so the points $P_{f,i} \colonequals [\beta_{f,i}:\alpha_{f,i}]$ are the roots of
$f$ in $\PP^1(\hatKf{f})$.

We now explain how to determine from $\hatKfm{f}{1}$ the cross ratios
$[P_{f,1},P_{f,2};P_{f,3},P_{f,j}]$ for $j=4,\dots,n$. (This is of course
vacuous for $n=3$.) First, we will apply the $\QQ$-linear projection
\[
	\rho_j : \mathcal S(f;\QQ) \rightarrow \mathcal S(f;\QQ), \qquad
	r\mapsto \tfrac{1}{4}(r-(1\;3).r-(2\;j).r+(1\;3)(2\;j).r)
\]
to the $\QQ$-vector space of elements on which the transpositions $(1\;3)$ and $(2\;j)$ both act as negation.

\begin{lemma}
  \label{lem:U-generators}
  Let $n\geq3$ and $f\in V_n'(\ZZ)$. For any $j=4,\dots,n$, the image of
  $\hatKfm{f}{1}$ under $\rho_j$ is the $\QQ$-vector space spanned by the
  expressions
  \[
    u_{f,j,S,T} \coloneq (\alpha_{f,1}\beta_{f,3}-\beta_{f,1}\alpha_{f,3})
    (\alpha_{f,2}\beta_{f,j}-\beta_{f,2}\alpha_{f,j}) \prod_{i\in S}\alpha_{f,i}\prod_{i\in T}\beta_{f,i}
\]
for partitions $\{1,2,3,j\}\sqcup S\sqcup T=\{1,\dots,n\}$.
\end{lemma}
\begin{proof}
  Look at the images of the monomials in $B_n(\ZZ)$ of degree $\leq1$ under the
  map $\rho_j$. Some are sent to zero (namely those in which $\beta_{f,1}$ and
  $\beta_{f,3}$ occur with the same exponent or in which $\beta_{f,2}$ and
  $\beta_{f,j}$ occur with the same exponent), and the others are sent to rational
  multiples of the claimed generators.
\end{proof}

Next, we apply the map
\[
  \pi_j:\hatKf{f}\rightarrow \hatKf{f} \times \hatKf{f},\qquad r \mapsto (r,(3\;j).r).
\]

\begin{lemma}
\label{lem:proj-to-dim2}
Let $n\geq3$ and $f\in V_n'(\ZZ)$. For any $j=4,\dots,n$, the
$\mathcal S(f;\QQ)$-vector space (not $\QQ$-vector space!) spanned by the image
of $\hatKfm{f}{1}$ under $\pi_j\circ\rho_j$ corresponds to the point
$(P_{f,1},P_{f,2};P_{f,3},P_{f,j}) \in \PP^1(\hatKf{f})$.
\end{lemma}
\begin{proof}
  The map $\pi_j$ sends $u_{f,j,S,T}$ to
\[
  \begin{pmatrix}
    (\alpha_{f,1}\beta_{f,3}-\beta_{f,1}\alpha_{f,3})
    (\alpha_{f,2}\beta_{f,j}-\beta_{f,2}\alpha_{f,j}) \\
    (\alpha_{f,1}\beta_{f,j}-\beta_{f,1}\alpha_{f,j})
    (\alpha_{f,2}\beta_{f,3}-\beta_{f,2}\alpha_{f,3})
  \end{pmatrix}
  \cdot \prod_{i\in S}\alpha_{f,i}\prod_{i\in T}\beta_{f,i}.
\]
This shows that the $\hatKf{f}$-span of $\pi_j(\rho_j(\hatKfm{f}{1}))$ is the line
in $\hatKf{f} \times \hatKf{f}$ corresponding to the cross ratio point
$(P_{f,1},P_{f,2};P_{f,3},P_{f,j}) \in \PP^1(\hatKf{f})$. (We have
$\alpha_{f,i},\beta_{f,i}\neq0$ for all $i$ as we assumed $f$ to be irreducible.)
\end{proof}

\begin{lemma}
\label{lem:U-equal-then-gl}
Let $n\geq3$ and let $f,f'\in V_n'(\ZZ)$ give rise to the same number field
$K\coloneq K_f = K_{f'}$. Assume that $\hatKfm{f}{1} = \hatKfm{f'}{1}$ (as subspaces
of $\hatK$). Then, $f$ and $f'$ are $G(\QQ)$-equivalent, i.e., there exist
$\lambda\in\QQ^\times$ and $\gamma\in\GL_2(\QQ)$ such that $\lambda f^\gamma = f'$.
\end{lemma}
\begin{proof}
  To simplify notation, we let
  $\hat K \colonequals \SnC{K}{\QQ} = \hatKf{f} = \hatKf{f'}$.

  The assumption $\hatKfm{f}{1} = \hatKfm{f'}{1}$ in particular implies
  $\pi_j(\rho_j(\hatKfm{f}{1})) = \pi_j(\rho_j(\hatKfm{f'}{1}))$ for
  $j=4,\dots,n$. By \Cref{lem:proj-to-dim2}, it follows that
  $(P_{f,1},P_{f,2};P_{f,3},P_{f,j}) = (P_{f,1}',P_{f,2}';P_{f,3}',P_{f,j}')$
  for $j=4,\dots,n$. Then, \Cref{lem:cross-ratio-map} shows that there is a
  matrix $\bar \gamma \in\PGL_2(\hat K)$ sending $P_{f,i}$ to $P_{f',i}$ for all $i$.

Then,
\[
  \tau(\bar\gamma) P_i = \tau(\bar\gamma) \tau\bigl(P_{\tau^{-1}(i)}\bigr) = \tau\bigl(\bar\gamma
  P_{\tau^{-1}(i)}\bigr) = \tau\bigl(P_{\tau^{-1}(i)}'\bigr) = P_i' = \bar\gamma P_i
\]
for all $\tau\in S_n=\Gal(\hat K|\QQ)$ and all $i$, in particular for
$i=1,2,3$. Since $\PGL_2(\hat K)$ acts freely on the set of triples of
(distinct) points in $\PP^1(\hat K)$, it follows that
$\tau(\bar \gamma)=\bar \gamma$ in $\PGL_2(\hat K)$ for all $\tau$. The long exact sequence in
Galois cohomology arising from the short exact sequence
\[
1\rightarrow\hat K^\times\rightarrow\GL_2(\hat K)\rightarrow\PGL_2(\hat K)\rightarrow1
\]
together with Hilbert's Theorem~90 shows that every
$\Gal(\hat K|\QQ)$-invariant element $\bar \gamma$ of $\PGL_2(\hat K)$ has a
$\Gal(\hat K|\QQ)$-invariant representative in $\GL_2(\hat K)$, i.e., a
representative $\gamma\in\GL_2(\QQ)$. Of course, this representative also sends
$P_i$ to $P_i'$. Then, \Cref{eq:concrete-factorization} implies
$\lambda f^{\gamma^{-1}} = f'$ for some $\lambda\in\QQ^\times$.
\end{proof}

Combining \Cref{lem:deg1-usually-equal,lem:U-equal-then-gl}, we immediately obtain:

\begin{proposition}
\label{thm:generically-injective}
For any $n\geq3$, there is a subset $V_n''(\ZZ)$ of $V_n'(\ZZ)$ of density $1$
such that any two elements of $V_n''(\ZZ)$ giving rise to the same number field
are $G(\QQ)$-equivalent.
\end{proposition}

The set $V_n'(\ZZ)$ only involves primitive forms, but we can also lift the
above result to $V_n(\ZZ)$.

\begin{theorem}\label{cor:generically-injective}
  For any $n\geq 3$, there is a subset $V_n'''(\ZZ)$ of $V_n(\ZZ)$ of density
  $1$ such that any two elements of $V_n'''(\ZZ)$ giving rise to the same
  number field are $G(\QQ)$-equivalent.
\end{theorem}

\begin{proof}
  By \Cref{thm:generically-injective} and \Cref{lem:usually-full-galois}, the
  set $V_n''(\ZZ)$ of \Cref{thm:generically-injective} has density $1$ among
  the set of primitive binary $n$-ic forms.

  Now let $V_n'''(\ZZ)$ be the set of binary $n$-ic forms that are equal to a
  nonzero integer times an element of $V_n''(\ZZ)$. Since, given two forms,
  whether they generate the same number field is invaiant under multiplying
  either form by a nonzero rational number, and the same is true for
  $G(\QQ)$-equivalence, any two elements of $V_n'''(\ZZ)$ giving rise to the
  same number field are $G(\QQ)$-equivalent.

  For each positive integer $m$, the density of $V_n'''(\ZZ)$ is at least the
  sum over $k$ from $1$ to $m$ of the density of elements of $V_n'''(\ZZ)$
  whose coefficients have greatest common divisor $k$. Elements of
  $V_n'''(\ZZ)$ whose coefficients have greatest common divisor $k$ are exactly
  the elements of $V_n''(\ZZ)$ times $k$, so their density is the density of
  $V_n''(\ZZ)$ times $k^{-(n+1)}$. The density of $V_n''(\ZZ)$ in $V_n(\ZZ)$ is
  equal to the density of primitive binary $n$-ic forms in $V_n(\ZZ)$ (since
  $V_n'(\ZZ)$ has density $1$ in binary $n$-ic forms), which is
  $\zeta(n+1)^{-1}$. Thus the density of $V_n'''(\ZZ)$ is at least
  $$ \frac{ \sum_{k=1}^m k^{-(n+1)}}{\zeta(n+1)}$$ and this converges to $1$ as
  $m$ goes to $\infty$ since the denominator converges to $\zeta(n+1)$, so the density
  of $V_n'''(\ZZ)$ is $1$. \end{proof}

\section{Counting equivalence classes of forms}
\label{sec:counting}
In this section we will prove the following theorem.

\begin{theorem}
\label{thm:counting-equivalence-classes}
For any $n\geq3$, there is a constant $C_n > 0$ such that the number $N(X)$ of
$G(\QQ)$-equivalence classes of binary $n$-ic forms containing a representative
$f\in V_n(\ZZ)$ with $\Ht(f) \leq X$ satisfies $N(X) \sim C_n\cdot X^{n+1}$ for $X\rightarrow\infty$.
\end{theorem}

As a consequence of this result, we will obtain \Cref{thm:main} for $n > 2$.

\begin{corollary}\label{thm:counting-fields-large-degree}
  For any $n\geq3$, there is a constant $C_n>0$ such that the number of
  (isomorphism classes of) degree $n$ number fields with $\PP^1$-height bounded
  by $X$ is
  \[N_n(\Ht < X) \colonequals \#\{K : [K:\QQ]=n,\,\Ht(K) \leq X\} \sim C_n \cdot X^{n+1}
\]
as $X\to\infty$.

\end{corollary}

\begin{proof}[Proof of \Cref{thm:counting-fields-large-degree}]
  Denote the set of (isomorphism classes of) degree $n$ number fields by
  $\FF_n$. Let $C_n$ as in \Cref{thm:counting-equivalence-classes}. Since to
  every $K\in\FF_n$ with $\Ht(K)\leq X$ corresponds at least one equivalence class
  of forms containing a representative of height at most $X$, this theorem
  immediately shows the upper bound $N_n(\Ht < X) \leq C_nX^{n+1} + o(X^{n+1})$.

  For the lower bound, note that by \Cref{thm:counting-equivalence-classes},
  there are $C_n X^{n+1} + o(X^{n+1})$ pairwise non-$G(\QQ)$-equivalent
  primitive binary $n$-ic forms of height at most $X$. (Clearly, any minimal
  height representative of an equivalence class is primitive.) By
  \Cref{lem:usually-full-galois}, only $o(X^{n+1})$ of these primitive forms do
  not lie in $V_n'(\ZZ)$. Only $o(X^{n+1})$ of the remaining ones do not lie in
  the set $V_n''(\ZZ)$ constructed in \Cref{thm:generically-injective}. By
  \Cref{thm:generically-injective}, the $C_nX^{n+1}+o(X^{n+1})$ remaining
  (pairwise non-$G(\QQ)$-equivalent) forms in $V_n''(\ZZ)$ correspond to
  distinct number fields. Hence, $N_n(\Ht < X) \geq C_nX^{n+1} + o(X^{n+1})$.
\end{proof}

\subsection{Proof of
  \texorpdfstring{\Cref{thm:counting-equivalence-classes}}{Theorem 7.1}} We
will prove \Cref{thm:counting-equivalence-classes} using a sieve over elements
of $G(\QQ)$. For the tail estimate, we will have to prove that for ``most''
elements $f\in V_n(\ZZ)$ (namely those not in $B(\varepsilon)\cup C(s)$), only a specific
finite set of elements $[(\lambda,\gamma)]\in G(\QQ)$ can send $f$ to an element
$\lambda f^\gamma$ of $V_n(\ZZ)$ with smaller height (\Cref{lem:finitely-many-matrices}).
For this purpose, we prove a generic upper bound on the ``absolute value'' of
$[(\lambda,\gamma)]$ in \Cref{lem:absolute-value-bound}, and a generic upper bound on the
``denominator'' of $[(\lambda,\gamma)]$ in \Cref{lem:denominator-bound}.

We denote by $|\gamma|$ the largest absolute value of an entry of $\gamma\in\GL_2(\RR)$.

\begin{lemma}
\label{lem:absolute-value-bound}
Let $f\in V_n(\RR)$ and $[(\lambda,\gamma)]\in G(\RR)$. Assume that
$\Ht(\lambda f^\gamma), \Ht(f)\leq X$. Then,
\[
|\lambda||\gamma|^{n-2}|{\det(\gamma)}|\cdot|\Disc(f)| \ll X^{2(n-1)},
\]
with the constant only depending on $n$.
\end{lemma}
\begin{proof}
  Replacing $\gamma$ by $\alpha\gamma\beta$ and $f$ by $f^{\alpha^{-1}}$ for appropriate orthogonal
  matrices $\alpha,\beta\in\operatorname{SO}_2(\RR)$, we can make $\gamma$ a diagonal matrix
  (singular value decomposition), say
  $\gamma = \left(\begin{smallmatrix}g_1\\&g_2\end{smallmatrix}\right)$ with
  $|g_1|\geq|g_2|$ and therefore $|\gamma|=|g_1|$. This substitution doesn't change
  $\lambda$, $\det(\gamma)$, or $\Disc(f)$. Since $\operatorname{SO}_2(\RR)$ is compact,
  it changes $|\gamma|$ and $\Ht(f)$ and $\Ht(\lambda f^\gamma)$ by bounded factors. Hence,
  after the substitution, we still have $\Ht(\lambda f^\gamma),\Ht(f) \ll X$ and it still
  suffices to show $|\lambda||\gamma|^{n-2}|{\det(\gamma)}|\cdot|\Disc(f)| \ll X^{2(n-1)}$.

The binary $n$-ic form
\[
  f=f_0x^n + f_1x^{n-1}y + \cdots + f_n y^n
\]
is sent to
\[
  \lambda f^\gamma = f_0\lambda g_1^nx^n + f_1\lambda g_1^{n-1}g_2 x^{n-1}y + \cdots + f_n\lambda g_2^ny^n.
\]
Looking at the first two coefficients, we see that
\begin{equation}
\label{eq:upper-bound-f0}
|f_0||\lambda||\gamma|^{n-2}|{\det(\gamma)}| = |f_0\lambda g_1^{n-1}g_2| \leq |f_0\lambda g_1^n| \leq \Ht(\lambda f^\gamma) \ll X
\end{equation}
and
\begin{equation}
\label{eq:upper-bound-f1}
|f_1||\lambda||\gamma|^{n-2}|{\det(\gamma)}| = |f_1\lambda g_1^{n-1}g_2| \leq \Ht(\lambda f^\gamma) \ll X.
\end{equation}
The discriminant of $f$ is a polynomial in the coefficients $f_0,\dots,f_n$ of
the form
\begin{equation}
\label{eq:disc-first-two-coeffs}
\Disc(f) = f_0 A(f_0,\dots,f_n) + f_1^2 B(f_0,\dots,f_n)
\end{equation}
with homogeneous polynomials $A,B\in\ZZ[F_0,\dots,F_n]$ of degrees $2n-3$ and
$2n-4$, respectively. (This can be seen for example by writing
$\pm n^{n-2}\Disc(f)$ as the determinant of the Sylvester matrix for the forms
$\frac{\partial}{\partial X}f(X,Y)$ and $\frac{\partial}{\partial Y} f(X,Y)$ and expanding the determinant
and using that all coefficients of the determinant polynomial are integers; see
\cite[Chapter~12, Equation~(1.31)]{GelfandKapranovZelevinsky}.) Hence,
\[
  |\Disc(f)| \ll \max(|f_0|,|f_1|)\Ht(f)^{2n-3}
  \stackrel{(\ref{eq:upper-bound-f0}), (\ref{eq:upper-bound-f1})} \ll
  \frac{X^{2n-2}}{|\lambda||\gamma|^{n-2}|{\det(\gamma)}|}. \qedhere
\]
\end{proof}

\begin{lemma}
\label{lem:denominator-bound}
Let $p$ be a prime number and let $f\in V_n(\ZZ_p)$ and
$[(\lambda,\gamma)]\in G(\QQ_p)$. Assume that $\lambda f^\gamma$ lies in
$V_n(\ZZ_p)$ and that the entries of $\gamma$ lie in $\ZZ_p$, but are not all
divisible by $p$. Then,
\[
v_p(\lambda\det(\gamma)) + 2\left\lfloor\frac{v_p(\Disc(f))}{2}\right\rfloor \geq 0.
\]
\end{lemma}
\begin{proof}
  Replacing $\gamma$ by $\alpha\gamma\beta$ and $f$ by $f^{\alpha^{-1}}$ for appropriate matrices
  $\alpha,\beta\in\GL_2(\ZZ_p)$, we can make $\gamma$ a diagonal matrix whose diagonal entries
  are powers of $p$ (Smith normal form), which by the assumptions on $\gamma$ means
  $\gamma = \left(\begin{smallmatrix}1\\&\det(\gamma)\end{smallmatrix}\right)$. This
  substitution does not change any of the assumptions or any of the valuations.

  The binary $n$-ic form
\[
  f=f_0x^n + f_1x^{n-1}y + \cdots + f_n y^n
\]
is sent to
\[
  \lambda f^\gamma = f_0\lambda x^n + f_1\lambda\det(\gamma) x^{n-1}y + \cdots + f_n\lambda \det(\gamma)^ny^n.
\]
Looking at the first two coefficients, we see that $\lambda f^\gamma\in V_n(\ZZ_p)$ implies
\[
  v_p(f_0) + v_p(\lambda) \geq 0
\]
and
\[
v_p(f_1)+v_p(\lambda)+v_p(\det(\gamma))\geq0.
\]
By (\ref{eq:disc-first-two-coeffs}) together with $f\in V_n(\ZZ_p)$, we have
\[
v_p(\Disc(f)) \geq \min(v_p(f_0), 2v_p(f_1)),
\]
so
\[
  \textnormal{(a)}\quad v_p(\Disc(f)) + v_p(\lambda) \geq 0 \qquad\textnormal{or}\qquad
  \textnormal{(b)}\quad \left\lfloor\frac{v_p(\Disc(f))}{2}\right\rfloor + v_p(\lambda) +
  v_p(\det(\gamma)) \geq 0.
\]

In case (a), the claim follows immediately unless $v_p(\Disc(f))$ is odd,
$v_p(\lambda)=-v_p(\Disc(f))$, and $v_p(\det(\gamma))=0$. But then $\lambda$ would lie in
$p^{-v_p(\Disc(f))}\ZZ_p^\times$ and $\gamma$ would lie in $\GL_2(\ZZ_p)$, so
$\lambda f^\gamma\in V_n(\ZZ_p)$ would imply
$f\in p^{v_p(\Disc(f))}V_n(\ZZ_p)$, which would mean that
$v_p(\Disc(f)) \geq v_p(\Disc(f))\cdot2(n-1) > v_p(\Disc(f))$.

In case (b), the claim follows immediately.
\end{proof}

\begin{lemma}
\label{lem:finitely-many-matrices}
Assume $n\geq3$. For every $\varepsilon>0$ and every integer $s\geq1$, there is a finite subset
$G'(\varepsilon,s)$ of $G(\QQ)$ such that for all
$f\in V_n(\ZZ)\setminus(B(\varepsilon)\cup C(s))$, the following holds: if
$[(\lambda,\gamma)]\in G(\QQ)$ with $\lambda f^\gamma\in V_n(\ZZ)$ and
$\Ht(\lambda f^\gamma)\leq \Ht(f)$, then $[(\lambda,\gamma)]\in G'(\varepsilon,s)$.
\end{lemma}
(As we have seen in \Cref{thm:quadratic-case},
\Cref{thm:counting-fields-large-degree} is wrong for $n=2$. This lemma is the
only place in the paper where we seriously use the assumption $n\geq3$.)
\begin{proof}
  Multiplying by an appropriate element of
  $T(\QQ) = \{(t^{-n},\left(\begin{smallmatrix}t\\&t\end{smallmatrix}\right))\mid
  t\in \QQ^\times\}$, we can make the entries of $\gamma$ relatively prime integers. Let
  $a\geq1$ be the largest integer such that $a^2\mid\Disc(f)$. By assumption,
  $a\leq s$ as $f\notin C(s)$. \Cref{lem:denominator-bound} implies
  $\lambda\det(\gamma)a^2 \in \ZZ$ and therefore
  $|\lambda\det(\gamma)|\geq a^{-2} \geq s^{-2}$. The assumption
  $|\Disc(f)|\geq\varepsilon\Ht(f)^{2(n-1)}$ as $f\notin B(\varepsilon)$ together with
  \Cref{lem:absolute-value-bound} for $X = \Ht(f)$ implies
  $|\lambda\det(\gamma)||\gamma|^{n-2}\ll\varepsilon^{-1}$ and therefore
  $|\gamma|^{n-2} \ll \varepsilon^{-1} s^2$. Since $n\geq3$, this leaves only finitely many options
  for the integer matrix $\gamma$. For any choice of $\gamma$, the conditions
  $\lambda\det(\gamma)a^2\in\ZZ$ with $1\leq a\leq s$ and
  $|\lambda\det(\gamma)||\gamma|^{n-2}\ll\varepsilon^{-1}$ leave only finitely many options for the
  rational number $\lambda$.
\end{proof}

\begin{corollary}\label{cor:equiv-class-size-bound}
  Any $G(\QQ)$-equivalence class of binary forms contains at most
  $|G'(\varepsilon,s)|$ representatives $f\in V_n(\ZZ)\setminus(B(\varepsilon)\cup C(s))$.
\end{corollary}
\begin{proof}
  The lemma shows that for any representative
  $f\in V_n(\ZZ)\setminus(B(\varepsilon)\cup C(s))$, there are at most
  $|G'(\varepsilon,s)|$ representatives
  $f'\in V_n(\ZZ)\setminus(B(\varepsilon)\cup C(s))$ with $\Ht(f') \leq \Ht(f)$.
\end{proof}

\begin{proof}[Proof of \Cref{thm:counting-equivalence-classes}]
  Order the elements of $V_n(\RR)$ by height, breaking ties lexicographically.
  This defines a total order $\leq$ on the set $V_n(\RR)$ and a well-order on
  $V_n(\ZZ)$. Let
\[
  A \coloneq \{f\in V_n(\ZZ) : \lambda f^\gamma\notin V_n(\ZZ)\textnormal{ or }\lambda f^\gamma \geq f \textnormal{
    for all }[(\lambda,\gamma)]\in G(\QQ)\},
\]
the set of forms $f\in V_n(\ZZ)$ which are minimal in their equivalence class. By
definition, $N(X) = N(A, X)$.

Let $\varepsilon>0$ and $s\geq1$. By \Cref{lem:finitely-many-matrices}, we have
$A \subseteq A'(\varepsilon,s)$ and $A'(\varepsilon,s)\setminus A \subseteq B(\varepsilon) \cup C(s)$, where
\[
  A'(\varepsilon,s) \coloneq \{f\in V_n(\ZZ) : \lambda f^\gamma\notin V_n(\ZZ)\textnormal{ or }\lambda f^\gamma \geq f
  \textnormal{ for all }[(\lambda,\gamma)]\in G'(\varepsilon,s)\}.
\]
In particular,
\[
N(A'(\varepsilon,s),X) - N(B(\varepsilon),X) - N(C(s),X) \leq N(A, X) \leq N(A'(\varepsilon,s),X).
\]

The set $A'(\varepsilon,s)$ is a finite union of polyhedral cones intersected with
congruence classes. Hence, there is a number $\alpha(\varepsilon,s)\geq0$ such that
\[
  N(A'(\varepsilon,s), X) = \alpha(\varepsilon,s) X^{n+1} + O_{\varepsilon,s}(X^n) \qquad\textnormal{for }X\rightarrow\infty.
\]
Recall \Cref{lem:usually-large-disc,lem:usually-almost-squarefree}:
\[
  N(B(\varepsilon), X) = \beta(\varepsilon)X^{n+1} + O_\varepsilon(X^n) \qquad\textnormal{for }X\rightarrow\infty
  \qquad\textnormal{with}\qquad\lim_{\varepsilon\to0} \beta(\varepsilon)=0,
\]
\[
  N(C(s), X) \leq \gamma(s)X^{n+1} + o_s(X^{n+1}) \qquad\textnormal{for }X\rightarrow\infty
  \qquad\textnormal{with}\qquad\lim_{s\to\infty} \gamma(s)=0.
\]

Combining the above facts and letting $X\rightarrow\infty$ and then
$\varepsilon\rightarrow0$ and $s\rightarrow\infty$, we conclude that
\[
N(X) = N(A,X) = C_n X^{n+1} + o(X^{n+1})
\qquad\textnormal{with}\qquad C_n = \lim_{\substack{\varepsilon\to0\\s\to\infty}} \alpha(\varepsilon,s) \geq0.
\]

To prove $C_n>0$, we use that any equivalence class in $V_n(\ZZ)$ has at most
$|G'(\varepsilon,s)|$ representatives
$f\in V_n(\ZZ)\setminus(B(\varepsilon)\cup C(s))$ by \Cref{cor:equiv-class-size-bound}. Hence,
\[
N(X) \cdot |G'(\varepsilon,s)|
\geq N(V_n(\ZZ), X) - N(B(\varepsilon), X) - N(C(s), X).
\]
Clearly,
\[
N(V_n(\ZZ), X) = (2X)^{n+1} + O(X^n)
\qquad\textnormal{for }X\rightarrow\infty.
\]
Picking $\varepsilon$ sufficiently small and $s$ sufficiently large so that
$\beta(\varepsilon) + \gamma(s) < 2^{n+1}$, and then letting $X\rightarrow\infty$, we conclude that
\[
C_n \geq \frac{2^{n+1} - \beta(\varepsilon) - \gamma(s)}{|G'(\varepsilon,s)|} > 0.
\qedhere
\]
\end{proof}

\begin{proof}[Proof of \Cref{thm:abstract}]
  Let $W(X)$ be the set of $f\in V_n(\ZZ)$ with $\Ht(f)\leq X$ for which there is
  some $g\in V_n(\ZZ)$ with $\Ht(g)\leq X$ such that $K_f \cong K_g$ although
  $f$ and $g$ are not $G(\QQ)$-equivalent. We need to show
  $|W(X)| = o(X^{n+1})$.

  We first show that with $V_n'''(\ZZ)$ as in \Cref{cor:generically-injective},
  we have
\begin{equation}\label{eq:w-containment}
W(X) \subseteq \bigcup_{\substack{g\in V_n(\ZZ)\setminus V_n'''(\ZZ):\\\Ht(g)\leq X}} \{ f \in V_n(\ZZ) : \Ht(f)\leq X \textnormal{ and } K_f \cong K_g \}.
\end{equation}
Indeed, take any $f, g \in V_n(\ZZ)$ with $\Ht(f),\Ht(g)\leq X$ satisfying
$K_f \cong K_g$. If $f$ does not lie in the set on the right-hand side, then both
$f$ and $g$ must lie in $V_n'''(\ZZ)$, so by \Cref{cor:generically-injective},
they must be $G(\QQ)$-equivalent. This shows \Cref{eq:w-containment}.

For any $\varepsilon>0$ and $s\geq1$, letting
\[
R(\varepsilon, s, g) \colonequals
\{ f \in V_n'''(\ZZ) \setminus (B(\varepsilon)\cup C(s)) : K_f \cong K_g \},
\]
we can bound the size of the set on the right-hand side of
\Cref{eq:w-containment} as follows:
\[
|W(X)|
\leq N(V_n(\ZZ) \setminus V_n'''(\ZZ), X) + N(B(\varepsilon), X) + N(C(s), X) + \sum_{\substack{g \in V_n(\ZZ)\setminus V_n'''(\ZZ):\\\Ht(g) \leq X}} |R(\varepsilon, s, g)|.
\]
Any two elements $f,f'$ of the same set $R(\varepsilon,s,g)$ are
$G(\QQ)$-equivalent by \Cref{cor:generically-injective} since they lie in
$V_n'''(\ZZ)$ and satisfy $K_f\cong K_{f'}$. By \Cref{cor:equiv-class-size-bound},
it follows that each of the sets $R(\varepsilon, s, g)$ has size at most
$|G'(\varepsilon, s)|$. Combining this with \Cref{cor:generically-injective} and
\Cref{lem:usually-large-disc,lem:usually-almost-squarefree}, we obtain the
upper bound
\[
|W(X)| \leq o(X^{n+1}) + (\beta(\varepsilon) X^{n+1} + O_\varepsilon(X^n)) + (\gamma(s) X^{n+1} + o_s(X^{n+1})) + o(X^{n+1})\cdot |G'(\varepsilon,s)|.
\]
Letting $X\to\infty$ and then $\varepsilon\to0$ and $s\to\infty$, we conclude that $|W(X)| = o(X^{n+1})$ as claimed.
\end{proof}

\section*{Acknowledgments}
This work was started as part of a workshop on ``Nilpotent counting problems in
arithmetic statistics'' hosted by the American Institute of Mathematics. We
thank the workshop organizers, Brandon Alberts, Yuan Liu, and Melanie Matchett
Wood, for bringing us together to work on this project. We would also like to
thank Henryk Iwaniec for help with locating \cite{Chelluri2004}.

\appendix

\bibliographystyle{amsalpha}
\providecommand{\bysame}{\leavevmode\hbox to3em{\hrulefill}\thinspace}
\providecommand{\MR}{\relax\ifhmode\unskip\space\fi MR }
\providecommand{\MRhref}[2]{%
  \href{http://www.ams.org/mathscinet-getitem?mr=#1}{#2}
}
\providecommand{\href}[2]{#2}

\end{document}